\documentclass[reqno]{amsart}
\usepackage{amsmath,amssymb}
\usepackage{graphicx,color}
\usepackage{amscd,amsthm}
\usepackage{array}
\usepackage[all]{xy}
\usepackage{color}
\usepackage{comment}
\usepackage{mathrsfs}

\newcommand{\LL}{\Lambda} %Novikov field
\newcommand{\LLo}{\Lambda_0} %Novikov ring
\newcommand{\LLp}{\Lambda_+} %maximal ideal of Novikov ring
\newcommand{\M}{\mathcal M} %moduli
\newcommand{\cM}{\overline{\mathcal M}} %compactified moduli
\newcommand{\lag}{\mathcal{L}} %object of \F

\newcommand{\g}{\frak{g}} %Lie algebra
\newcommand{\bb}{\frak{b}} %bounding cochain
\newcommand{\xx}{{\underline{X}}} %vector field associated to an element of Lie algebra
\newcommand{\F}{\mathcal{F}} %Fukaya cat
\newcommand{\ob}{\operatorname{Ob}} %objects
\newcommand{\crit}{\operatorname{Crit}} %critical points
\newcommand{\mf}{\operatorname{MF}} %matrix factorizations
\newcommand{\br}{\operatorname{Br}} %formal matrix factorizations

\newcommand{\ev}{\operatorname{ev}} %evaluation
\newcommand{\m}{\frak{m}} %A_\infty operations

\newcommand{\val}{\operatorname{val}} %valuation
\newcommand{\bcar}{\mathrm{Car}} %Cartan model

  \newtheorem{theorem}{Theorem}[section]
  \newtheorem{definition}[theorem]{Definition}
  \newtheorem{remark}[theorem]{Remark} 
  \newtheorem{lemma}[theorem]{Lemma}
  \newtheorem{proposition}[theorem]{Proposition}
  \newtheorem{corollary}[theorem]{Corollary}

  \newtheorem{assumption}[theorem]{Assumption}

\newcommand{\Z}{\mathbb{Z}}

\newcommand{\C}{\mathbb{C}}
\newcommand{\R}{\mathbb{R}}

\newcommand{\mr}{\mathrm{m}}
\newcommand{\aaa}{\mathfrak{a}}
\newcommand{\an}{\mathrm{an}}
\newcommand{\ano}{\mathcal{O}}
\newcommand{\bv}{\vec{b}}
\newcommand{\ini}{\mathrm{in}}
\newcommand{\Hom}{\mathrm{Hom}}
\newcommand{\Spec}{\mathrm{Spec}}
\newcommand{\Spm}{\mathrm{Spm}}
\newcommand{\lan}[1]{\langle #1 \rangle}

\newcommand{\trop}{\mathrm{trop}}
\newcommand{\Trop}{\mathrm{Trop}}
\newcommand{\inte}{{\mathrm{int}}}

\newcommand{\uP}{\mathcal{U}_P}
\newcommand{\pr}{\operatorname{pr}}

\newcommand{\perf}{\mathrm{perf}}
\newcommand{\pe}{\mathrm{pe}}
\newcommand{\stab}{\mathrm{stab}}
\newcommand{\sing}{\mathrm{Sing}}
\newcommand{\cl}{Cl}
\newcommand{\op}{\mathrm{op}}
\newcommand{\module}{\mathrm{mod}}
 
\newcommand{\uotimes}{\, \underset{\C}{\otimes}\, }
\newcommand{\sg}{\wedge \g ^\vee}

\newcommand{\Sym}{\mathrm{Sym}}

\newcommand{\A}{\mathcal{A}} %gapped filtered A_\infty category
\newcommand{\cA}{{\overline{\mathcal{A}}}} %A_\infty category
\newcommand{\cF}{\overline{\mathcal{F}}} 

\newcommand{\hhom}{\mathcal{H}\mathit{om}}
\newcommand{\cZ}{\mathbb{Z}} %integers
\newcommand{\cR}{\mathbb{R}} 
\newcommand{\Lag}{\mathbb{L}}

\begin{document}
\title{Equivariant Homological Mirror Symmetry for $\C$ and $\C P^1$}
\author{Masahiro Futaki}
\address{Department of Mathematics and Informatics, Graduate School of Science, Chiba University, 1-33, Yayoicho, Inage-ku, Chiba-shi, Chiba, 263-8522 Japan}
\email{futaki@math.s.chiba-u.ac.jp}
\author{Fumihiko Sanda}
\address{Department of Mathematics, Graduate School of Science, Kyoto University, Kitashirakawa-Oiwake-cho, Sakyo-ku, Kyoto, 606-8502, Japan}
\email{sanda.fumihiko.5s@kyoto-u.ac.jp}
\date{\today}

\begin{abstract}
In this paper we define an equivariant Floer $A_\infty$ algebra 
for $\C$ and $\C P^1$ 
by using Cartan model.
We then prove an equivariant homological mirror symmetry, 
i.e.\ an equivalence between an $A_\infty$ category of equivariant Lagrangian branes 
and the category of matrix factorizations of Givental's equivariant Landau-Ginzburg potential function.
\end{abstract}
\maketitle
\tableofcontents

\renewcommand{\thefootnote}{}
\footnote{
M.~F.\ is supported by Grant-in-Aid for Scientific Research 
(C) (18K03269) 
of the Japan Society for the Promotion of Science. 
F.~S.\ is partially supported by Grant-in-Aid for Scientific Research 
(B) (17K17817), 
(S) (16H06337), 
(S) (21H04994) 
of the Japan Society for the Promotion of Science.
}

\section{Introduction}

One form of the homological mirror symmetry for a toric Fano manifold $X$ relates 
the Fukaya category of $X$ to the dg-category of matrix factorizations 
of its Landau-Ginzburg mirror $F_X$, e.g.\ \cite{CHLtor} \cite{ELGmfd}.
In this paper we study an equivariant version of this story.

Givental introduced the equivariant Landau-Ginzburg potential $F_X$ mirror to 
the toric manifold $X$, 
e.g.\ \cite{giv}.
The potential is of the form
\begin{align}
F_X &= f_X - \lambda_1 \log g_1 - \cdots - \lambda_r \log g_r , \nonumber
\end{align}
where $f_X$ is the non-equivariant Landau-Ginzburg mirror potential of $X$ and 
$\lambda_1 ,..., \lambda_r$ are equivariant parameters 
corresponding to a basis of the Lie algebra of the torus $T$ acting on $X$, 
and $g_1,\ldots ,g_r$ are some invertible functions.

Cho-Oh \cite{CO} defined a potential function by using Floer theory and showed that 
it coincides with the (nonequivariant) Landau-Ginzburg potential 
in the case of compact toric Fano manifolds 
based on the idea of Hori-Vafa \cite{HV}, 
cf \cite{MR2057871} \cite{fooo:toric1}.
Kim-Lau-Zheng \cite{KLZpot} showed that Givental's equivariant Landau-Ginzburg mirror 
is recovered from the $T$-equivariant potential function  
when $X$ is a semi-projective semi-Fano toric manifold.

The equivariant Fukaya category is yet to be defined, 
but should consist of the Lagrangians 
preserved by the group action 
and of the equivariant Floer $A_\infty$ algebras.

There have been several approaches to the equivariant Floer theory.
Zernik \cite{zerequ, zerloc} defined an equivariant Floer cohomology 
for possibly nonorientable Lagrangians using the Cartan model, 
and applied it to the study of the open Gromov-Witten theory 
of the real projective space inside the complex projective space.
Kim-Lau-Zheng mentioned above employed the Morse model 
to study the disc potential.

In this paper we introduce our version of the equivariant Floer $A_\infty$ algebra 
using the Cartan model, 
and prove the following.
\vspace{5pt}

\begin{theorem}[Section \ref{equivhms}, 
Theorems \ref{maintheorem1} and \ref{maintheorem2}]

Let $X$ be $\C P^1$ or $\C$ with the standard $S^1$ action.
We have a cohomologically fully faithful $A_\infty$ functor from $\F_X$ to $\br(F_X)$ 
whose image split-generates the triangulated category $[\br(F_X)]$, 
where $\F_X$ is an equivariant Floer $A_\infty$ category 
and $\br(F_X)$ is the dg-category of matrix factorizations of $F_X$.
\end{theorem}

\noindent
For precise statements, see Section \ref{equivhms}.

This paper is organized as follows.
In Section \ref{algprelim}, 
we review algebraic notions: 
$\g$-differential spaces and its $A_\infty$ version, 
the Cartan models and the gapped filtered $A_\infty$ categories.
In Section \ref{equivfloer}, 
we review some Floer theory, and then define 
an equivariant Floer $A_\infty$ algebra 
and an $A_\infty$ category of equivariant Lagrangian branes.
In Section \ref{mf}, 
we define the dg-category of matrix factorizations 
for Givental type potential functions.
In Section \ref{equivhms}, 
we formulate and prove the main theorems.
In Appendix A, we compute the dimension of the Jacobian ring of the equivariant Landau-Ginzburg mirror of a semi-projective toric manifold (Theorem \ref{maintheorem}).
In Appendix B, we introduce another version of  categories of branes for equivariant Landau-Ginzburg potentials.
\vspace{5pt}

\noindent 
{\bf Notations.} \\
$\LL := \{ \sum_{i \geq 0} a_i T^{\lambda_i} | 
a_i \in \C, 
\lambda_0 < \lambda_1 < \cdots \in \R, 
\lim_{i\to\infty}\lambda_i = \infty \}$ the Novikov field over $\C$, \\
$\LLo := \{ \sum_{i \geq 0} a_i T^{\lambda_i} \in \LL | 
\lambda_i \geq 0 \} 
\subset \LL$ the Novikov ring, 
and\\
$\LLp := \{ \sum_{i \geq 0} a_i T^{\lambda_i} \in \LL | 
\lambda_i > 0 \} \subset \LLo$ the maximal ideal of the Novikov ring. \\
We define the valuation of elements of the Novikov field by its $T$-exponent 
of the initial term. 
Namely, 
$\val (a) := \lambda_0$ for $a=\sum_{i \geq 0} a_i T^{\lambda_i} \in \LL^\ast$ 
with $a_0 \neq 0$ and set $\val(0) := \infty$.\\
We also use the following notations:\\
$\hat\uotimes \, $ : the $\Z/2\Z$-graded completed tensor product over $\C$.\\
$|x|$ : the degree of a homogeneous element of a graded module.\\
$|x|' := |x| -1$ : the shifted degree.\\
$\Spec$ : the set of prime ideals with Zariski topology.\\
$\Spm$ : the set of maximal ideals with Zariski or analytic topology.\\
$\mathrm{Sym}$ : the graded symmetric tensor product.\\
$[1]$ : the degree one shift of a graded module, i.e., for a graded module $V$,  the degree $k$ part $({V[1]})^k$ is $V^{k+1}$.

\section{Algebraic preliminaries}\label{algprelim}

In this section, let $(\g, [ \cdot, \cdot])$ be a complex Lie algebra (considered as a $\Z$-graded Lie algebra concentrated in degree $0$).
\subsection{Equivariant cohomology of $\g$-differential spaces}\label{sub 2.1}
In this subsection, we recall some basics on de Rham models of equivariant cohomology.
See, e.g., \cite{guillemin} for more details.
Let $(M,\delta)$ be a cochain complex, i.e., $M$ is a $\Z$-graded vector space over $\C$ and $\delta$ is a degree $1$ endmorphism of $M$ with $\delta^2=0$.
\begin{definition}
Let $(M, \delta)$ be a cochain complex.
Suppose that two linear maps $i$ and $L$ of degree $-1$ and $0$ respectively are given
\begin{align*}
\g \underset{\C}{\otimes} M &\to M : X\otimes x  \mapsto i_X x\\
\g \underset{\C}{\otimes} M &\to M : X\otimes x  \mapsto L_X x
\end{align*}
which satisfy the following $(X, Y \in \g)$: 
\begin{align*}
\delta L_X-L_X \delta&=0, \\
\delta i_X + i_X \delta&=L_X, \\
L_X  L_Y-L_Y L_X&=L_{[X, Y]}, \\
L_X i_Y-i_YL_X&=i_{[X, Y]}, \\
i_X i_Y+i_Yi_X&=0.
\end{align*}
Then $(M, \delta, L, i)$ is called a $\g$-differential space $($we denote it briefly by $M)$:
\end{definition}
For $\g$-differential spaces $M, N$, we easily see that $M \underset{\C}{\otimes} N$ is naturally a $\g$-differential space.
For a $\g$-differential space $M$, we set 
\begin{align*}
{M}^\g&:=\{x \in M \mid L_Xx=0 \ \text{for all}\ X \in \g\},\\
{M}_\mathrm{hor}&:=\{x \in M \mid i_Xx=0 \ \text{for all}\ X \in \g\}, \\
{M}_\mathrm{bas}&:=\{x \in M \mid L_Xx=0, i_Xx=0 \ \text{for all}\ X \in \g\}.
\end{align*}
Then $M_\mathrm{hor}$ is closed under $L_X$, and $M^\g$ and $M_\mathrm{bas}$ are closed under $\delta.$
For a $\C$-linear subspace $N \subseteq M,$ the intersection $N \cap M^\g$ is also denoted by $N^\g$. 

Let $\wedge \g$ be the ($\Z$-graded) symmetric algebra $\mathrm{Sym}(\g [1])$, $\wedge \g^\vee$ be $\mathrm{Sym}(\g^\vee[-1])$, and $S\g^\vee$ be $\mathrm{Sym}(\g^\vee[-2])$.
Choose a basis $e_1, e_2, \dots, e_r$ of $\g$ and let $e^1, e^2, \dots, e^r$ be the dual basis of $\g^\vee$. 
Let $c^i_{j k}$ be the structure constants of $\g$, i.e., $c^i_{j k}:=\lan{e^i, [e_j, e_k]}$.
There exists a natural non-degenerate bilinear paring $\lan{\cdot, \cdot} : \sg \uotimes \wedge \g \to \C$. 
For $X_1, \dots, X_k \in \g, x_1, \dots , x_k \in \g^\vee$, 
we have
\[\lan{x^1\wedge \cdots \wedge x^k, X_1\wedge \cdots \wedge X_k}
=\mathrm{det}((\lan{x^i, X_j})_{i, j=1}^k).\] 
The element $e^i \in \g ^\vee$ is also denoted by $\theta^i$ if it is considered as an element of $\sg$. 

%%%%%%%%%%%%%%%                   g-DIFFERENTIAL SPACE STRUCTURE ON EXTERIOR ALGEBRA      %%%%%   
We introduce a structure of a $\g$-differential space on $\sg$.
We define 
\[ \delta \colon \wedge \g^\vee \to \wedge \g^\vee \]
by the formula 
$\delta (\theta^i)=-\frac{1}{2}\sum_{j, k}c^i_{j k} \theta^j \theta^k$.
(Note that this $\delta$ comes from the dual of the Lie bracket.)

For $X, Y \in \g$ and $x \in \g^\vee$, set 
\[\lan{L_Xx, Y}:=-\lan{x, [X, Y]}, \ i_Xx=\lan{X, x}.\]
Extending $L_X$ and $i_X$ by the Leibniz rule, we obtain two linear maps \[L_X, i_X : \sg \to \sg.\]
We see that $(\sg, \delta, L, i)$ is a $\g$-differential space.
We note that $\delta$ is also written as $\frac{1}{2}\sum_i \theta^i \circ L_{e_i}$, where $\theta^i$ is the left multiplication by $\theta^i.$ 

%%%%%%%%%%%%%%%%                           WEIL ALGEBRA            %%%%%%%%%%%%%%%%%%%%%%%%%%%%%%%%%%%
We next introduce the Weil algebra.
Set $W\g:=S \g^\vee \underset{\C}{\otimes} \wedge \g^\vee,$ which is naturally a $\Z$-graded commutative algebra.
For simplicity, $1 \otimes e^i \in W\g$ is also denoted by $\theta^i$ and $e^i \otimes 1 \in W\g$ is denoted by $\overline{e}^i$.
Set $F^i:=\overline{e}^i+\frac{1}{2}\sum_{j, k}c^i_{j k} \theta^j \theta^k$,
then $\theta^1, \dots, \theta^r, F^1, \dots, F^r$ also generate $W\g$.
We note that $|\theta ^i|=1, |\overline{e}^i|=2, $ and $|F^i|=2$.
We define 
\[\delta_\mathrm{W}(\theta^i)=\overline{e}^i=F^i-\frac{1}{2} \sum_{j, k} c^i_{j k} \theta^j \theta^k \in W \g, \ 
  \delta_\mathrm{W}(\overline{e}^i)=0,\]
and extend it to $\delta_\mathrm{W} : W \g \to W \g$ by the Leibniz rule.
From the definition, we easily see that $(W\g, \delta_\mathrm{W})$ is acyclic.
By using the Jacobi identity, we have 
\[\delta_\mathrm{W}(F^i)=\sum_{j, k} c^i_{j k} F^j \theta^k.\]
Similarly, we define two linear maps $L, i : \g \underset{\C}{\otimes} W \g \to W \g$ by
\begin{align*}
i_{e_j}(\theta^i)&:=\delta^i_j, \hspace{1.85cm} i_{e_j}(F^i):=0, \\
L_{e_j}(\theta^i)&:=-\sum_k c^i_{j k}\theta^k, \ L_{e_j}(F^i):=-\sum_kc^i_{j k}F^k.
\end{align*}
Then $(W\g, \delta_\mathrm{W}, L, i)$ gives a $\g$-differential space, which is called a Weil algebra.
By definition,  we see that $(W \g)_\mathrm{hor} \cong S \g^\vee,$ 
which is freely generated by $F^1, F^2, \dots, F^r \in (W \g)_\mathrm{hor}$ and closed under $L_X$.

Let $M$ be a $\g$-differential space.
Then $M \underset{\C}{\otimes} W\g$ is also a $\g$-differential space and its differential is also denoted by $\delta_\mathrm{W}$.
Set $C^*_\mathrm{W}(M):=(M \underset{\C}{\otimes} W\g)_\mathrm{bas}.$
This is a $\Z$-graded $(W \g)_\mathrm{bas} \cong (S\g^\vee)^\g$-module and $\delta_\mathrm{W}$ is compatible with this module structure.
\begin{definition}\label{Weil}
The cochain complex $(C^*_\mathrm{W}(M), \delta_\mathrm{W})$ is called a Weil model
and its cohomology is called an equivariant cohomology of $M$.
The equivariant cohomology is an $(S\g^\vee)^\g$-module which is denoted by $H^*_\g(M)$.
\end{definition}

We next define the Cartan model of equivariant cohomology of a $\g$-differential space $M$.
Set $C^*_\mathrm{Car}(M):=(M \underset{\C}{\otimes}(W \g)_\mathrm{hor})^\g, $ which is a $\Z$-graded $(S \g^\vee)^\g$-module.
To simplify notation, for $x \in W \g$, the left multiplication by $x$ is also denoted by $x$.
We define a differential $\delta_\mathrm{Car} : C^*_\mathrm{Car}(M) \to C^*_\mathrm{Car}(M)$ by 
\[\delta_\mathrm{Car}:=\delta \otimes 1-\sum_i i_{e_i}\otimes F^i,\]
which is compatible with the $\Z$-graded $(S\g^\vee)^\g$-module structure.
\begin{definition}\label{CAR} 
The cochain complex $(C^*_\mathrm{Car}(M), \delta_\mathrm{Car})$ is called a Cartan model.
\end{definition}
Set $\gamma:=\sum_j i_{e_j} \otimes \theta^j \in \mathrm{End}(M \underset{\C}{\otimes} W\g).$
This is a degree $0$ nilpotent operator and we define an automorphism $\phi:=\mathrm{exp}(\gamma)$ (called a Mathai-Quillen morphism).
This morphism $\phi$ is compatible with the $\Z$-graded $W\g$-module structure on $M \uotimes W\g$.
\begin{theorem}[See, e.g., {\cite[Theorem 3.2]{kalkman}} and {\cite[Chapter 4]{guillemin}}]\label{KAL}
The image of the Weil model $C^*_\mathrm{W}(M)$ by $\phi$ is $C^*_\mathrm{Car}(M)$
and $\delta_\mathrm{Car} \circ \phi=\phi \circ \delta_\mathrm{W}$.
Hence $\phi$ gives an isomorphism between the equivariant cohomology $H^*_\g(M)$ and the cohomology of the Cartan model as $(S\g^\vee)^\g$-modules.
\end{theorem}

%%%%%%%%%%%%%%%%%%%%%%%%             A INFINITY CATEGORIES        %%%%%%%%%%%%%%%%%%%%%%%
\subsection{Preliminaries on gapped filtered $A_\infty$ categories}\label{sub 2.3}
Let $\mathrm{G}$ be a discrete submonoid of $2\cZ \times \cR_{\geq 0}$. 
We denote by $\mu \colon \mathrm{G} \to 2\Z$ the first projection and 
by $\omega \colon \mathrm{G} \to \R_{\geq 0}$ the second projection.
Suppose that for each $E \in \R_{\geq 0}$
\[\# \{ \beta \in\mathrm{G} \mid \omega(\beta) \leq E \} < \infty.\]
Let $R$ be a $\Z$-graded commutative algebra over $\C$.
We first recall the notion of an $A_\infty$ category (over $R$).
 \begin{definition}
 A $\Z$-graded unital $A_\infty$ category $(\cA, \{\m^\cA_k\})$ over $R$ 
 consists of the following data $($the first three data is simply denoted by $\cA)$:
 \begin{itemize}
 \item $\ob\cA$: a set of objects,
 \item $\cA(A,B)$: $\Z$-graded modules over $R$ 
 for pairs $(A,B) \in \left(\ob\cA\right)^2$,
 \item $1_A$: degree $0$ elements of $\cA(A, A)$ for $A \in \ob\cA$.
 \item $R$-module maps $\m_k^\cA \ (k \geq 1)$ of degree $1$ 
 for $(k+1)$-tuples $(A_0,...,A_k) \in \left(\ob\cA\right)^{k+1}$,
 \begin{align}
 \m_k^\cA
 &\colon \cA(A_0,A_1)[1]\underset{R}\otimes\cdots\underset{R}\otimes\cA(A_{k-1},A_k)[1]\to\cA(A_0,A_k)[1] \nonumber
 \end{align}
 \end{itemize}
 such that they satisfy the following relations:
 \begin{itemize}
 \item $A_\infty$ relations : 
 \begin{align}
 \sum_{\substack{k_1+k_2=k+1\\1 \leq k_1, k_2\\0\leq i< i+k_2\leq k}}
 (-1)^{\#}\m^\cA_{k_1}(x_1,...,x_i,\m^\cA_{k_2}(x_{i+1},...,x_{i+k_2}),x_{i+k_2+1},...,x_k) = 0 \nonumber
 \end{align}
 where $x_i$ are homogeneous elements of $\cA(A_{i-1}, A_i)$ and $\# = |x_1|'+\cdots+|x_i|'$.
 \item Unitality:
 \begin{align}
 \m_2^\cA(1_A,x)=(-1)^{|x|}\m_2^\cA(x,1_A)=x &{\rm \ if \ }k=2, \nonumber\\
 \m_k^\cA(...,1_A,...)=0 &{\rm \ if \ }k\neq 2. \nonumber
 \end{align}
 \end{itemize}
 For simplification, a $\Z$-graded unital $A_\infty$ category $(\cA, \{\m^\cA_k\})$ is also denoted by $\cA$.
 \end{definition}
 The morphisms $\m_k^\cA$ naturally induce degree $2-k$ morphisms
 \begin{align}
 &\cA(A_0,A_1)\underset{\C}\otimes\cdots\underset{\C}\otimes\cA(A_{k-1},A_k)\to\cA(A_0,A_k) \label{unshifted}
 \end{align}
 which are also denoted by $\m_k^\cA.$
 
 We next recall the notion of a unital $\mathrm{G}$-gapped filtered $A_\infty$ category.
\begin{definition}[cf.\,\cite{fukaya:mir2}]
A unital $\mathrm{G}$-gapped filtered $A_\infty$ category $(\cA, \{\m_{k, \beta}\})$ consists of the following data:
\begin{itemize}
\item $\cA$: a unital $\Z$-graded $A_\infty$ category over $R$.
\item   $R$-module morphisms of degree $1-\mu(\beta)$ with $\m_{k, 0}=\m_k  ($especially $\m_{0, 0}=0)$
for $A_0, \dots, A_k \in \ob\cA, k \in \Z_{\ge0}, \beta \in \mathrm{G}$:
\begin{align}
\m_{k, \beta} \colon \cA(A_0,A_1)[1]\underset{R}\otimes\cdots\underset{R}\otimes\cA(A_{k-1},A_k)[1]\to\cA(A_0,A_k)[1], \nonumber
\end{align}
which naturally induce degree $2-k-\mu(\beta)$ morphisms $\m_{k, \beta}$ 
similar to $($\ref{unshifted}$)$,
\end{itemize}
such that they satisfy the following relations:
\begin{itemize}
\item
For $A_0,...,A_k \in \ob\cA $ and homogeneous elements $x_i \in \cA(A_{i-1}, A_i),$ the morphisms $\m_{k, \beta}$ satisfy the $A_\infty$ relations for $(k, \beta) \in \Z_{\ge 0} \times \mathrm{G}$:
\begin{align}
\sum_{\substack{k_1+k_2=k+1, \ 0 \leq k_1, k_2, \\0\leq i\leq k-k_2\\
\beta_1+\beta_2=\beta}}
(-1)^{\#}\m_{k_1,\beta_1}(x_1,...,x_i,
\m_{k_2,\beta_2}(x_{i+1},...,x_{i+k_2}),x_{i+k_2+1},...,x_k) = 0 \nonumber
\end{align}
where $\# = |x_1|'+\cdots+|x_i|'$.
\item 
The units $1_A$ satisfy 
\begin{align}\label{unit}
&\m_{k,\beta}(\cdots,1_A,\cdots)=0 &{\rm if\ } (k,\beta)\neq (2,0),\\ \label{unit2}
&\m_{2,0}(1_A,x)=(-1)^{|x|}\m_{2,0}(x,1_A)=x &{\rm if\ } (k,\beta)=(2,0).
\end{align}
\end{itemize}
We simply say an unital gapped filtered $A_\infty$ category when we don't specify $\mathrm{G}$.
\end{definition}
To simplify notation $\m_{1, 0}$ is also denoted by $\delta$, then we have $\delta^2=0$. 

\begin{remark}
We obtain a $\Z/2\Z$-graded unital curved $A_\infty$ category $(\A, \{\m_k\})$ over $\Lambda$ 
from a gapped filtered $A_\infty$ category $(\cA,\{\m_{k,\beta}\})$ over $\C$ by taking
\begin{align}
\ob\A &:= \ob\cA, \nonumber\\
\A(A,B) &:= \cA(A,B)\hat\uotimes\Lambda,\nonumber\\
1_A&:=1_A \in \A (A, A),\nonumber\\
\m_k &:=\sum_{\beta\in\mathrm{G}} T^{\omega(\beta)/2\pi} \m_{k,\beta}.\nonumber
\end{align} 
\end{remark}
%%%%%%%%%%%%%%%%%%%%%%%              G ACTION                %%%%%%%%%%%%%%%%%%%%%%%%%%
\subsection{$\g$-differential gapped filtered $A_\infty$ categories}\label{sub 2.4}
Let $\g$ be a complex Lie algebra.
\begin{definition}\label{gfA}
Let $(\cA, \{\m_{k, \beta}\})$ be a unital $\mathrm{G}$-gapped filtered $A_\infty$ category over $\C$.
Suppose that $\cA(A, B)$ is a $\g$-differential space with the differential $\delta=\m_{1,0}$ for each $A, B \in \ob\cA$.
These data is called a $\g$-differential unital $\mathrm{G}$-gapped filtered $A_\infty$ category over $\C$
if they satisfy the following equation for each $k \in \Z_{\ge 0}$ and $\beta \in \mathrm{G}$ with $(k, \beta) \neq (1, 0)$: 
\begin{align}\label{interior}
   i_X \m_{k,\beta} (x_1,...,x_k) + \sum_{i=1}^k (-1)^{|x_1|'+\cdots +|x_{i-1}|'} \m_{k,\beta} (x_1,..., i_X x_i ,..., x_k) = 0.
\end{align}
Here $X \in \g$ and $x_i \in \cA(A_{i-1}, A_i)$ are homogeneous elements.
\end{definition}
\begin{proposition}
For $k \in \Z_{\ge 0}, \beta \in \mathrm{G}, X \in \g$, and homogeneous elements $x_i \in \cA(A_{i-1}, A_i) \ (i=1,2, \dots , k)$, we have 
\begin{align}\label{lieder}
 L_X \m_{k,\beta} (x_1,...,x_k)=\sum_{i=1}^k \m_{k,\beta} (x_1,..., L_X x_i ,..., x_k).
\end{align} 
\end{proposition}
\begin{proof}
For $(k, \beta)=(1, 0)$, this proposition follows from the definition of $\g$-differential space.
We assume $(k, \beta) \neq (1, 0)$.
Set 
\begin{align*}
       \delta(x_1 \otimes \cdots \otimes x_k)&:=
       \sum_{i=1}^k(-1)^{|x_1|'+ \cdots +|x_{i-1}|'} x_1 \otimes \cdots \otimes \delta x_i \otimes \cdots \otimes x_k \\
       i_X(x_1 \otimes \cdots \otimes x_k)&:=
       \sum_{i=1}^k(-1)^{|x_1|'+ \cdots +|x_{i-1}|'} x_1 \otimes \cdots \otimes i_X x_i \otimes \cdots \otimes x_k 
 \end{align*}
 We also set       
 \[ \hat{\m}^2_{k, \beta}(x_1\otimes \cdots \otimes x_k):=
       \sum(-1)^{|x_1|'+ \cdots +|x_i|'}\m_{k_1, \beta_1}(x_1 \otimes \cdots \otimes 
       \m_{k_2, \beta_2}(x_{i+1} \otimes \cdots \otimes x_{i+k_2}) \otimes \cdots \otimes x_k), \]
where the sum is taken over the set 
\[\{(k_1, k_2, \beta_1, \beta_2, i)\mid \substack{k_1+k_2=k+1, \ \beta_1+\beta_2=\beta, \ 0 \le i \le k-k_2, \\  
      (k_1, \beta_1) \neq (1, 0), \ (k_2, \beta_2)\neq (1, 0)}\}.\]
By definition, we have $\delta \circ \m_{k, \beta}+\m_{k, \beta} \circ \delta + \hat{\m}_{k, \beta}^2=0.$
Then we see that 
\begin{align*}
      \delta i_X \m_{k, \beta}&=-\delta \m_{k, \beta} i_X=\m_{k, \beta} \delta i_X+\hat{\m}_{k, \beta}^2 i_X, \\
      i_X \delta \m_{k, \beta}&=-i_X \m_{k, \beta} \delta-i_X \hat{\m}_{k, \beta}^2=\m_{k, \beta} i_X \delta-\hat{\m}_{k, \beta}^2 i_X.
      \end{align*}
 Hence we have $L_X \m_{k, \beta}=\m_{k, \beta}(\delta i_X +i_X \delta).$
 Combined with the equation
 \[(\delta i_X +i_X \delta)(x_1 \otimes \cdots \otimes x_k)=\sum_{i=1}^k x_1 \otimes \cdots \otimes L_Xx_i \otimes \cdots \otimes x_k,\]
 we obtain the desired equation.
\end{proof}

Since $i_X \circ \m_{2, 0}+\m_{2, 0} \circ (i_X \otimes \mathrm{id} + \mathrm{id} \otimes i_X)=0$ for $X \in \g$, we have
\begin{align}
i_X(1_A)=0, \ L_X(1_A)=(\delta \circ i_X+i_X \circ \delta)(1_A)=0 . \label{dunit}
\end{align}

Let $(\cA, \{\m_{k, \beta}\})$ be a $\g$-differential unital $\mathrm{G}$-gapped filtered $A_\infty$ category over $\C$.
Let $\bullet$ be $\mathrm{W}$ or $\mathrm{Car}$.
By Equations (\ref{dunit}), we have $1_A \in C^0_\bullet(\cA(A, A)).$

For $A_0, A_1, \dots, A_k \in \ob\cA$ and $(k, \beta) \neq (1, 0)$, 
let $\m_{k, \beta}^\bullet$ be the trivial extension of $\m_{k, \beta}$ to 
\[
C^*_\bullet(\cA(A_0, A_1))[1] \underset{(S\g^\vee)^\g}{\otimes} \cdots \underset{(S\g^\vee)^\g}{\otimes} C^*_\bullet(\cA(A_{k-1}, A_k))[1],\]
i.e., for homogeneous elements $x_i \in \cA(A_{i-1}, A_i)$ and $f_i \in W\g$ 
\[\m^\bullet_{k, \beta}(x_1\otimes f_1, \dots, x_k \otimes f_k):=(-1)^{\sum_{i<j}|f_i|\cdot|x_j|'}
  \m_{k, \beta}(x_1, \dots, x_k)\otimes (f_1\cdots f_k)\]
and set $\m_{1, 0}^\bullet:=\delta_\bullet$.

Then $C_\bullet(\cA)$ consists of the following data:
\begin{itemize}
\item a set of objects $\ob C_\bullet(\cA):=\ob \cA.$
\item $\Z$-graded $(S\g^\vee)^\g$-modules $C_\bullet(\cA)(A, B):=C^*_\bullet(\cA(A, B))$ for $A, B \in \ob\cA.$
\item degree $0$ morphisms $1_A \in C_\bullet(\cA)(A, A)$ for $A \in \ob\cA$.
\item the homomorphisms $\m_{k,\beta}^\bullet$.
\end{itemize}

By using Equations $(\ref{interior})$, $(\ref{lieder})$, and $(\ref{dunit})$, we easily see that $\m_{k, \beta}^\bullet$ give a
unital $\mathrm{G}$-gapped filtered $A_\infty$ algebra structure over $(S\g^\vee)^\g$ on $C_\bullet(\cA).$
Moreover the Mathai-Quillen morphism $\phi$ satisfies
\begin{align*}
\phi\circ\m_{k, \beta}^\mathrm{W}&=\m_{k, \beta}^\mathrm{Car}\circ(\overbrace{\phi \otimes\cdots\otimes\phi}^k) 
\end{align*}
i.e., $\phi$ gives an $A_\infty$ functor.

Suppose that $\g$ is abelian.
Let $\lambda$ be a $\C$-algebra homomorphism from $(W\g)_{\mathrm{hor}} \cong S\g^\vee$ to $\Lambda_0,$
which is called an equivariant parameter.
By evaluating $\m_{k, \beta}^\mathrm{Car}$ at $\lambda$, we will define a $\Z/2\Z$-graded unital curved $A_\infty$ category 
$(\A^\g, \{\m_k^\lambda\})$ over $\Lambda.$
We define $\cA^\g$ by
\begin{itemize}
\item a set of objects $\ob \cA^\g:=\ob \cA,$
\item $\Z$-graded modules $\cA^\g(A, B):=\cA(A, B)^\g$,
\item degree $0$ elements $1_A \in \cA^\g(A, A)$ (recall Equations $(\ref{dunit})$).
\end{itemize}
By Equation (\ref{lieder}), the restrictions of $\m_{k, \beta}$ to the $\g$-invariant part give a unital $\mathrm{G}$-gapped filtered $A_\infty$ structure $\m_{k, \beta}^\g$ on $\cA^\g.$
This unital $\mathrm{G}$-gapped filtered $A_\infty$ algebra induces a 
$\Z/2\Z$-graded unital curved $A_\infty$ structure $\m_k^\g$ on $\A^\g$,
where $\A^\g(A, B):=\cA^\g(A,B)\hat{\uotimes}\Lambda$ and $\m_k^\g:=\sum_{\beta \in \mathrm{G}}
T^{\frac{\omega(\beta)}{2\pi}} \m^\g_{k, \beta}.$ 
Set 
\begin{align*}
\m_k^\lambda:=
\begin{cases}
\m_k^\g &\text{ if } k \neq 1 \\
\m_1^\g- \sum_{i=1}^r \lambda(F^i) i_{e_i} & \text{ if } k=1.
\end{cases}
\end{align*}
By Equations (\ref{interior}) and $(\ref{dunit})$,
we see that $(\A^\g, \{\m_k^\lambda\})$ is a $\Z/2\Z$-graded unital curved $A_\infty$ category over $\Lambda.$
We note that this $\Z/2\Z$-graded curved $A_\infty$ category is induced from a ($\Z$-graded) $\mathrm{G}'$-gapped filtered $A_\infty$ category
$(\cA^\g, \{\m_{k, \beta}^\lambda\})$ over $\C$ for some $\mathrm{G}' \supseteq \mathrm{G}$. 

%%%%%%%%%%%%%%%%%%%%           MC TWISTS                  %%%%%%%%%%%%%%%%%%%%%%%%%%%%%%%%%%%%%%%%%%
We introduce a notion of  bounding cochains for unital $\g$-differential gapped filtered $A_\infty$ categories.
Let $(\cA, \{\m_{k, \beta}\})$ be a unital $\g$-differential $\mathrm{G}$-gapped filtered $A_\infty$ category over $\C$ and
$\mathfrak{b}_{+, i}$ are odd elements of $\cA(A_i, A_i) \hat{\uotimes}\Lambda_+ \ (A_0, \dots, A_k \in \ob \cA)$.
We recall that the following operators
\begin{align}\label{mctwist}
\m^{\mathfrak{b}_+}_k(x_1,...,x_k) 
:=\sum_{l_0+\cdots+l_k=l} 
\m_{k+l}
(\overbrace{\bb_{+,0},...,\bb_{+, 0}}^{l_0},x_1,
\overbrace{\bb_{+, 1},...,\bb_{+,1}}^{l_1},...,x_k,
\overbrace{\bb_{+,k},...,\bb_{+, k}}^{l_k})
\end{align}
also satisfy the $A_\infty$ relations and unitality with units $1_{A_0}, \dots, 1_{A_k}$(cf. \cite[Proposition 1.20]{fukaya:mir2}).  

Suppose that 
\begin{align}
\mathfrak{b}_{+, i} \in \cA^1(A_i, A_i)\hat{\uotimes}\Lambda_+ \nonumber
\end{align}
and there exist $c_0, \dots, c_k \in \Hom_\C(\g, \Lambda_+)$ such that 
\begin{align}
i_X(\mathfrak{b}_{+, 0})=c_0(X) \cdot 1_{A_0}, \dots,\  i_X(\mathfrak{b}_{+, k})=c_k(X) \cdot 1_{A_k} \nonumber
\end{align}
for all $X \in \g$ and $i=0, 1, \dots, k$, where $\cA^1(A_i, A_i)$ denotes the degree one part of $\cA(A_i, A_i)$.
If $c_0=c_1=\cdots =c_k$, then, by using Equations (\ref{unit}) and (\ref{unit2}), we easily see that the operators $\m^{\mathfrak{b}_+}_k$ also satisfy
the Equation (\ref{interior}) (for some monoid $\mathrm{G}'$).
Hence, by choosing a finite collection $\{(A_i, \bb_{+, i})\} \ (i=0, 1, \dots, k)$ as objects of $\cA^\mathrm{bc}$
and putting 
\begin{align}
\cA^\mathrm{bc}((A_i, \bb_{+, i}), (A_j,\bb_{+, j}))=
\begin{cases} \cA(A_i, A_j) & \text{if} \ c_i=c_j \\ 0 &\text{if}\  c_i \neq c_j, \nonumber
\end{cases}
\end{align}
we obtain a unital $\g$-differential gapped filtered $A_\infty$ category $(\cA^{\mathrm{bc}}, \{\m_{k, \beta}^{\bb_+}\})$.
\begin{definition}\label{bc}
Let $(\cA, \{\m_{k, \beta}\})$ be a unital $\g$-differential $\mathrm{G}$-gapped filtered $A_\infty$ category over $\C.$
Let $A \in \ob \cA$ and $\mathfrak{b_+}$ be an element of $\cA^1(A, A) \hat\uotimes \Lambda_+$ with $\delta \bb_+=0$. 
Suppose that $i_X(\mathfrak{b}_+) \in \Lambda_+ \cdot 1_A$ for all $X \in \g$.
The element $\mathfrak{b}_+$ is called a bounding cochain if $\m_0^{\bb_+}(1) \in \Lambda_+ \cdot 1_A.$
\end{definition}

Choose an equivariant parameter $\lambda$ and suppose that $\bb_{+, i}$ are bounding cochains.
Then we have $L_X \bb_{+, i}=0$ and we can deform $(\cA^\g, \{\m_{k, \beta}^\lambda\})$ by the same way as Equation (\ref{mctwist}).
Thus we obtain a unital gapped filtered  $A_\infty$ category $(\cA^\g, \{\m_{k, \beta}^{\bb_+, \lambda}\})$
(for some monoid $\mathrm{G}'$).
Let $(\A^\g, \{\m_k^{\bb_+, \lambda}\})$ be the associated curved $A_\infty$ category.
By the definition, the curvature term $\m^{\mathfrak{b_+}, \lambda}_0(1) \in \A^\g(A_i, A_i)$ is 
\begin{align}\label{curvature}
\sum_{k=0}^\infty \m_{k}^\g(\bb_{+, i},\dots,\bb_{+, i})
- \sum_{j=1}^r \lambda(F^j) i_{e_j}(\bb_{+, i})  . 
\end{align}
Since $\bb_{+, i}$ is a bounding cochain, we have 
\[\m^{\mathfrak{b_+}, \lambda}_0(1) \in \Lambda_+ \cdot 1_{A_i}.\]
 Finally, we obtain a $\Z/2\Z$-graded (uncurved) unital $A_\infty$ category $(\A'^{\g}, \{ \m'^{\bb_+, \lambda}_k \})$ over $\Lambda$
 by setting 
 \begin{align*}
 \A'^\g((A_i, \bb_{+, i}), (A_j, \bb_{+, j}))=
 \begin{cases}
 \A^\g((A_i, \bb_{+, i}), (A_j, \bb_{+, j}))  \ &\text{if} \ \m^{\bb_{+, i}, \lambda}_0(1) =  \m^{\bb_{+, j}, \lambda}_0(1) \\
 0 \ &\text{if} \ \m^{\bb_{+, i}, \lambda}_0(1) \neq  \m^{\bb_{+, j}, \lambda}_0(1)
 \end{cases}
 \end{align*} 
 and
 \begin{align*}
 \m'^{\bb_+, \lambda}_k(x_1, x_2, \dots, x_k)=
 \begin{cases}
 \m^{\bb_+, \lambda}_k(x_1, x_2, \dots, x_k) &\text{if} \ k \ge 1 \ \text{and} \ x_1 \neq 0 , \cdots, x_k \neq 0\\
 0 &\text{if} \ k=0 \ \text{or} \ x_i=0 \ \text{for some} \ i.
 \end{cases}
 \end{align*}
 Here $x_i$ are homogeneous elements of morphism spaces of $\A'^\g.$ 

\section{Equivariant Floer theory}\label{equivfloer}

In this section we shall formulate 
a version of the equivarinat Floer theory 
using the de Rham model.
We first review some basics on the equivariant cohomology 
to fix conventions. 

\subsection{Equivariant cohomology}\label{equivcoh}

Let $G$ be a compact connected Lie group, 
$Lie(G)$ be its Lie algebra and 
$\g := Lie(G) \otimes_\R \C$ be its complexification.
We take a basis $e_1,...,e_r$ of $Lie(G)$, 
regarded also as the basis of $\g$ over $\C$ 
and its dual basis $e^1,...,e^r$ of $\g^\vee$.
Then $G$ acts on $S\g^\vee$ via  
$(g\cdot f)(X) := f(g^{-1}\cdot X)$ where $g \in G$ and $X \in \g$, 
and $g^{-1} \cdot X$ denotes the adjoint action of $g^{-1}$.

Let $M$ be an $n$-dimensional $G$-manifold, i.e.\ $G$ acts smoothly on $M$ from the left 
and let $\rho$ be a $\C$-local system on $M$.
For $X \in \g$, $\xx \in \Gamma (TM)$ is defined by 
$\left.\dfrac{d}{dt}\right|_{t=0} \exp (tX)\cdot p$ for $p \in M$.

\begin{definition}
Let $\Omega^\ast(M;\rho)$ be the de Rham complex with coefficients in $\rho$ 
with the differential locally defined as 
$d(\alpha\otimes s) := d\alpha \otimes s$, 
where $\alpha$ is a complex valued form and $s$ is a flat section of $\rho$.
We simply denote by $\Omega^\ast(M)$ the de Rham complex 
of $\C$-valued differential forms with the trivial local system of rank $1$.

Let $i_\xx \colon \Omega^\ast(M;\rho) \to \Omega^{\ast-1}(M;\rho)$ 
be the interior product locally defined by $i_\xx(\alpha\otimes s) := i_\xx (\alpha)\otimes s$ 
for $X \in \g$, 
and let $L_\xx \colon \Omega^\ast(M;\rho) \to \Omega^{\ast}(M;\rho)$ 
be the Lie derivative $L_\xx := d i_\xx + i_\xx d$.
Then the quadruple $(\Omega^\ast(M;\rho),\delta,L,i)$ forms 
a $\g$-differential space, 
where $\delta := d$, $L_X := L_\xx$ and $i_X := i_\xx$.
\end{definition}

The equivariant cohomology of $M$ is defined to be 
the cohomology of the Weil model of this $\g$-differential space:
$H^\ast_G(M) := H^\ast(\Omega^\ast_W(M),\delta_W)$.
See Atiyah-Bott \cite[Theorem 4.13]{ab} for the relationship between the equivariant cohomologies 
defined by the Weil model and by the homotopy quotient 
$EG \underset{G}\times M$.

Let $M$ be a compact manifold with corners.
There are several different formulations of manifolds with corners.
In this paper we use the formulation by Joyce \cite{joyce}.
We call $M$ a $G$-manifold with corners if it is a manifold with corners 
equipped with a smooth $G$-action.
When $M$ is oriented, $\partial M$ can be equipped with an orientation 
so that the orientation of the both sides of the following coincide:
$T_xM \cong \R v \oplus T_{x'} \partial M$ 
where $x' \in \partial M$, $x = i(x')$ via $i \colon \partial M \to M$ 
and $v$ is an outward vector at $x$ (cf.\ \cite[Convention 7.2]{joyce}).
The smooth differential $r$-forms on $M$ with coefficients in $\rho$ 
are defined to be the smooth sections $M \to \bigwedge^r T^\ast M \otimes \rho$.
We denote by $\eta|_{\partial M}$ the pull-back of a form $\eta$ on $M$ 
via $i \colon \partial M \to M$.

Let $M$ and $N$ be compact $G$-manifolds with corners, 
$f \colon M \to N$ be a smooth $G$-equivariant map 
and $\rho$ be a $\C$-local system on $N$.
Note that 
$f^\ast (i_\xx \eta) = i_\xx (f^\ast \eta)$, 
$f^\ast (L_\xx \eta ) = L_\xx (f^\ast \eta)$ hold 
and therefore 
$f$ induces a homomorphism of complexes 
$\Omega^\ast(N;\rho)^\g \to \Omega^\ast(M;f^\ast \rho)^\g$ 
which we also denote by $f^\ast$ abusing notation.

Assume further that $f$ is submersive 
(i.e.\ $f$ restricted to any stratum is submersive, \cite[Definition 3.2 (iv)]{joyce})
and that $M$ and $N$ are both oriented. 
We orient the fibers of $f$ 
such that the orientation of 
$T_{f(p)}N \oplus T^{\mathrm{fiber}}_{p}M$ 
coincides with that of $T_p M$, 
where $T^{\mathrm{fiber}}_{p}M$ is the tangent space of the fiber $f^{-1}(f(p))\subseteq M$ at $p \in M$ of $f$.
Then the integration along the fiber $f_\ast \eta$ 
of the form $\eta$ on $M$ with coefficients in $f^\ast \rho$ 
is defined as a form on $N$ with coefficients in $\rho$.
This is characterized by that the formula 
\begin{align}
\int_N \omega \wedge f_\ast \eta 
= \int_M f^\ast \omega \wedge \eta \nonumber
\end{align}
holds for any $\omega \in \Omega^\ast(N;\rho^\vee)$ 
(note that $\omega\wedge f_\ast\eta\in\Omega^\ast(N;\rho^\vee\otimes\rho )
\cong\Omega^\ast(N),
f^\ast\omega\wedge\eta\in\Omega^\ast(M;f^\ast\rho^\vee\otimes f^\ast\rho)
\cong\Omega^\ast(M)$).

\begin{lemma}
$i_\xx (f_\ast \eta) =  
f_\ast (i_\xx \eta)$. 
\end{lemma}
The proof is based on the local calculation.

This lemma implies that 
$L_\xx f_\ast \eta = (d i_\xx + i_\xx d)(f_\ast \eta ) = f_\ast (L_\xx \eta)$.
Therefore $f$ induces $\C$-linear maps between the spaces of $\g$-invariant forms 
$f_\ast \colon \Omega^\ast (M;f^\ast\rho)^\g 
\to \Omega^{\ast -\dim M +\dim N}(N;\rho)^\g$.
The restriction 
$f|_{\partial M} \colon \partial M \stackrel{i}\to M \stackrel{f}\to N$ is also submersive 
when $f$ is submersive, 
and the integrations along the fiber both for $f$ and for $f|_{\partial M}$ 
are defined.
Then the following holds.
\begin{lemma}
\begin{enumerate}
\item The Stokes formula holds: 
$df_\ast \eta =  
f_\ast d\eta + 
(-1)^{|\eta|+\dim M}
\left( f|_{\partial M} \right)_\ast \eta|_{\partial M}
$.\\
\item Let $L$ be another compact oriented manifold with corners 
and $g \colon N \to L$ be 
a smooth submersive map.
Then the integration along the fiber is compatible with composition, i.e.\ 
$(g \circ f)_\ast \eta = g_\ast \circ f_\ast \eta$.\\
\item $f_\ast (f^\ast \omega \wedge \eta ) = \omega \wedge f_\ast \eta$. 
\end{enumerate}
\end{lemma}
Let $M, N$ be compact oriented manifolds with corners, 
$L$ be a closed oriented manifold with a $\C$-local system $\rho$ 
and $f \colon M \to L$, $g \colon N \to L$ be submersions.
Recall that we can form the fiber product 
\[
\xymatrix{
 & N \underset{L}\times M \ar[dl]_t \ar[dr]^s & \\
 N \ar[dr]_g & & M \ar[dl]^f \\
 & L & 
 } 
 \]
so that $s$ and $t$ are smooth submersions.
The fiber product $N \underset{L}\times M$ is orientable 
and we define the orientation of the main stratum as follows:
decompose the tangent spaces of $N$ and $M$ at interior points as 
$TN = \ker dg \oplus g^\ast TL$ and $TM = f^\ast TL \oplus \ker df$, 
and define the orientation of the fiber product 
by the decomposition 
$T (N \underset{L}\times M) = t^\ast\ker dg \oplus t^\ast g^\ast TL \oplus s^\ast \ker df$.
Then we have the following formula.
\begin{lemma}
For $\eta \in \Omega^\ast(M;f^\ast\rho)$, 
\begin{align}
g^\ast \circ f_\ast \eta = t_\ast \circ s^\ast \eta . \nonumber
\end{align}
\end{lemma}

%%% Floer theory %%%
\subsection{Floer theory}\label{floertheory}

In this subsection we review some basics on the Floer theory.

Let $(M,\omega)$ be a compact symplectic manifold with 
$\dim M = 2n$ and $L$ be a compact oriented spin Lagrangian submanifold in $M$.
Choose an almost complex structure $J$ compatible with $\omega$.
For $k \in \mathbb{Z}_{\geq 0}$ and $\beta \in H_2(M,L;\mathbb{Z})$, 
we denote by 
$u \colon (D,\partial D ;z_0,...,z_k) \to (M,L)$ 
a $J$-holomorphic map $u \colon D \to M$ 
such that $u(\partial D) \subset L$, 
$u_\ast ([D,\partial D]) = \beta$ 
with $(k+1)$ boundary marked points $z_0,...,z_k \in \partial D$ 
which have counter-clockwise order.
We say two such maps $u \colon (D,\partial D ;z_0,...,z_k) \to (M,L)$, 
$u' \colon (D,\partial D ;z'_0,...,z'_k) \to (M,L)$ are equivalent 
if there exists a biholomorphic map $\phi \colon D \to D$ such that 
$u' = u \circ \phi$ and $\phi (z'_i) = z_i$, 
and denote by $u \sim u'$.
We then define the moduli space 
of $(k+1)$-pointed $J$-holomorphic disks bounded by $L$ to be 
the set of equivalence classes
\begin{align}
\M_{k+1,\beta}(L) := 
\{ u \colon (D,\partial D ;z_0,...,z_k) \to (M,L) \} / \sim \nonumber
\end{align}
and denote by $\cM_{k+1,\beta}(L)$ its compactification consisting of stable maps 
when $\beta \neq 0$ or $k \geq 2$.

\begin{proposition}[{\cite[Proposition 7.1.1]{fooo:book2}}]\label{boundary}
The moduli space $\cM_{k+1,\beta}(L)$ has a Kuranishi structure with an orientation 
of dimension $(n+k+\mu(\beta)-2)$ 
where $\mu(\beta)$ denotes the Maslov index of the class $\beta$, and 
we have an isomorphism of the spaces with Kuranishi structures with orientations
\begin{align}
\partial \cM_{k+1,\beta}(L) &\cong \coprod_{k_1+k_2=k+1, \beta_1+\beta_2=\beta} 
(-1)^{\sharp_1} \cM_{k_1+1,\beta_1}(L) {}_{\ev_i} \times_{\ev_0} \cM_{k_2+1,\beta_2}(L) 
\label{boundary}
\end{align}
where $\sharp_1 := k_1k_2 + ik_2 + i + n$ 
and the sum is taken over $(k,\beta)$ satisfying $\beta \neq 0$ or $k \geq 2$.
\end{proposition}
\begin{remark}
The orientation of the moduli space 
is determined by the orientation and the spin structure of $L$ \cite[Chapter 8]{fooo:book2}.
\end{remark}
\begin{definition}[cf. {\cite[Section 7]{fukaya:cyc}}]
\label{ainfty}
We can define the following operators.
\begin{align}
\m_{k,\beta}(x_1,\ldots,x_k) &:= (-1)^{\sharp_2} {\ev_0}_\ast 
(\ev_1 \times \cdots \times \ev_k)^\ast (x_1\times\cdots\times x_k), 
\label{floermk}\\
\intertext{
for $\beta \neq 0$ or $k\geq 2$, and
}
\m_{1,0}(x_1) &:= d x_1 \nonumber
\end{align}
for $(k,\beta)=(1,0)$, 
where 
$x_1,\ldots,x_k \in \Omega^\ast(L)$ and 
$\displaystyle \sharp_2 = 
\sum_{i=1}^k i |x_i|' + 1$.
\end{definition}
\begin{theorem}[{\cite[Theorem 7.1]{fukaya:cyc}}]\label{fukayacat}
$(\Omega^\ast(L),\{\m_{k,\beta}\})$ forms 
a unital $\mathrm{G}$-gapped filtered $A_\infty$ algebra for some $\mathrm{G}$.
The constant function $1 \in \Omega^0(L)$ gives the unit.
\end{theorem}

The pushforward $\ev_{0\ast}$ 
in the right hand side of (\ref{floermk}) 
is defined using the CF perturbation.
It can however be calculated 
using ordinary pullback and pushforward 
under the following assumption.
\begin{assumption}\label{regularity}
All the moduli spaces concerned in the definition of the $A_\infty$ structure 
are manifolds with corners, and the evaluation maps are submersions.
\end{assumption}

\begin{remark}\label{maslov2} 
The condition that the resulting form in (\ref{floermk}) has nonnegative degree is 
$\deg x_1 +\cdots + \deg x_k +2-\mu(\beta)-k \geq 0$.
Therefore only holomorphic disks with Maslov index $\mu(\beta)\leq 2$ contribute 
the $A_\infty$ structure when $n=1$.
All the holomorphic disks bounded by a moment fiber in a toric manifold 
are classified by Cho-Oh \cite{CO}.
In Section \ref{equivhms} we only consider the cases 
$M = \C P^1$ and $\C$, 
and in these cases the disk with Maslov index less than or equal to $2$ is 
either a constant disk (then the Maslov index is $0$) 
or a disk with Maslov index $2$.
Since there is no nonconstant holomorphic sphere with $c_1 \leq 0$, 
the moduli space $\cM_{k+1,\beta}(L)$ 
compactified with stable disks becomes 
a compact manifold with corners when $\mu(\beta)\leq 2$, 
and the evaluation maps are submersions.
Hence Assumtion \ref{regularity} is satisfied.
\end{remark}

We also have the following formula.
\begin{proposition}[Divisor axiom. {\cite[Proposition 6.3]{cho}}, see also {\cite[Proof of Lemma 11.8]{fooo:toric1}}]
Let $\theta \in \Omega^1(L)$ such that $d\theta=0$, then 
\begin{align}
\sum_{i=1}^k \m_{k,\beta}
(x_1,\ldots,x_{i-1},\theta,x_i,\ldots,x_{k-1}) 
&= \langle \partial \beta,\theta \rangle \m_{k-1,\beta}(x_1,\ldots,x_{k-1}) \label{divisor}
\end{align}
if $\beta \neq 0$ or $k \geq 2$.
\end{proposition}

%%% Equivariant Floer theory %%%
\subsection{Equivariant Floer theory}\label{equivfloertheory}
We now proceed to the construction of 
an equivariant Floer $A_\infty$ algebra.
Let $M$ be a compact smooth $G$-manifold 
with a $G$-invariant symplectic structure $\omega$ 
and a $G$-invariant $\omega$-compatible almost complex structure, 
$L$ be its compact oriented spin Lagrangian submanifold preserved by the $G$-action.
From now on we assume Assumption \ref{regularity}.
We first see the following.
\begin{proposition}\label{groupaction}
Under Assumption \ref{regularity}, 
we have the following formulae for $k \geq 2$ or $\beta \neq 0$.
\begin{align}
i_\xx \m_{k,\beta} (x_1,\ldots,x_k) + 
\sum_{i=1}^k(-1)^{|x_1|'+\cdots+|x_{i-1}|'}
\m_{k,\beta}(x_1,\ldots,i_\xx x_i,\ldots,x_k) =0 \nonumber
\end{align}
\end{proposition}

\proof 
\begin{align}
i_\xx \m_{k,\beta} (x_1,\ldots,x_k) 
&= (-1)^{\sharp_2}\ev_{0\ast}(\ev_1\times \cdots \times \ev_k)^\ast
(i_\xx(x_1\times \cdots \times x_k)) \nonumber\\
&= \sum_{i=1}^k (-1)^{\sharp_2+|x_1|+\cdots+|x_{i-1}|}
\ev_{0\ast}(\ev_1\times \cdots \times \ev_k)^\ast 
(x_1\times \cdots \times i_\xx x_i \times 
\cdots \times x_k) \nonumber\\
&= \sum_{i=1}^k(-1)^{\sharp_2+|x_1|+\cdots+|x_{i-1}|
+\sharp_2-i}
\m_{k,\beta}(x_1,\ldots, i_\xx x_i ,\ldots,x_k) \nonumber\\
&= -\sum_{i=1}^k(-1)^{|x_1|'+\cdots+|x_{i-1}|'}
\m_{k,\beta}(x_1,\ldots, i_\xx x_i ,\ldots,x_k) .\nonumber
\end{align}
\qed

\begin{corollary}
$\Omega^\ast(L)$ with the Floer $A_\infty$ structure 
$\{ \m_{k,\beta} \}$ 
becomes a unital $\g$-differential gapped filtered $A_\infty$ algebra over $\C$ 
in the sense of Definition \ref{gfA}.
\end{corollary}

As discussed in Section \ref{algprelim}, 
its Cartan model 
$(\Omega^\ast_\bcar(L),\{ \m_{k,\beta}^\bcar \})$ becomes 
a unital gapped filtered $A_\infty$ algebra.

The last step is to define the $\Z/2\Z$-graded filtered $A_\infty$ algebra 
over the Novikov field by substituting equivariant parameters.
From now on we assume that $G$ is a compact torus of dimension $r$.
%%% DEFINITION OF THE EQUIVARIANT FLOER ALGEBRA %%%
\begin{definition}\label{eqF}
Let $\lambda \colon S\g^\vee \to \Lambda_0$ be an equivariant parameter.
We call $(\Omega^\ast (L)^\g \hat\uotimes \Lambda , \m_k^\lambda )$ 
the equivariant Floer $A_\infty$ algebra of $L$, 
where $\hat\uotimes$ denotes the $\Z/2\Z$-graded completed tensor.
This is a unital $\Z/2\Z$-graded gapped filtered $A_\infty$ algebra over $\Lambda$ 
with the unit $1 \in \C \subseteq \Omega^0(L)^\g$.
\end{definition}

%%% bounding cochain %%%
Lastly we introduce the deformation of an equivariant Floer $A_\infty$ algebra 
by a bounding cochain, slightly generalizing the construction 
of Section \ref{algprelim} (see also \cite[Section 4]{fooo:toric1}).

%%% deformation by a bounding cochain %%%
\begin{definition}\label{boundingc}
Take $\bb = \bb_0 + \bb_+$ with $\bb_0 \in \Omega^1(L;\C)$ 
and $\bb_+ \in \Omega^1(L;\C) \hat\uotimes \LLp$ 
satisfying $d\bb =0$, 
and set $\rho \colon \pi_1(L) \to \C^\ast$ to be 
$\rho (\gamma) := e^{\langle \bb_0,\gamma \rangle}$ for $\gamma \in \pi_1(L)$.
We define the operators $\m_k^{\bb,\lambda}$ deformed by $\bb$ 
by the following formula:
\begin{align}
\m_k^{\bb,\lambda}(x_1,...,x_k) :=& 
\sum_\beta \rho(\partial\beta)T^{\omega(\beta)/2\pi}
\sum_{l_0+\cdots+l_k=l}
\m_{k+l,\beta}
(\overbrace{\bb_+,\ldots,\bb_+}^{l_0},x_1,
\ldots,x_k,
\overbrace{\bb_+,...,\bb_+}^{l_k}) \nonumber\\
&- \left\{ \begin{array}{lc}
0 & k \neq 0,1\\
\sum_{j=1}^r\lambda(F^j)i_{e_j}(x_1) & k = 1\\
\sum_{j=1}^r\lambda(F^j)i_{e_j}(\bb) & k=0
\end{array} \right. \label{evaluate}
\end{align}
We call $\bb = \bb_0+\bb_+$ a bounding cochain of $L$ 
if it satisfies the following:
\begin{align}
i_X \bb = i_X \bb_0 + i_X \bb_+ \in \Lambda_0 \cdot 1, \qquad
& \m_0^{\bb,\lambda}(1) \in \Lambda_0 \cdot 1. \label{mc}
\end{align}
The second equation is called the weak Maurer-Cartan equation.

If $\bb=\bb_0+\bb_+$ is a bounding cochain, 
$\{ \m_k^{\bb,\lambda} \}$ gives a structure of 
a unital gapped filtered $A_\infty$ algebra 
on $\Omega^\ast(L)^\g \hat\uotimes \Lambda$ 
which satisfies 
$\m_1^{\bb,\lambda} \circ \m_1^{\bb,\lambda} = 0$, 
i.e.\ we can define its cohomology with respect to the differential 
$\m_1^{\bb,\lambda}$.
\end{definition}

Using (\ref{divisor}) 
the curvature of this deformed $A_\infty$ structure is calculated as
\begin{align}
\m_0^{\bb,\lambda} (1) &= 
\sum_{\beta\neq 0} 
T^{\omega(\beta)/2\pi} 
e^{\langle \bb,\partial \beta\rangle}
\m_{0,\beta}(1) 
-\sum_{j=1}^r\lambda(F^j)i_{e_j}(\bb). \nonumber
\end{align}

Take $\rho$ a $\C$-local system of rank $1$ on $L$ which is expressed as 
$\rho(\gamma) = e^{\langle\bb_0,\gamma\rangle}$ for all $\gamma \in \pi_1(L)$ 
with some closed $1$-form $\bb_0 \in \Omega^1(L;\C)$, 
and $\bb_+\in\Omega^1(L;\C)\hat\otimes\LL_+$, such that 
$\bb=\bb_0+\bb_+$ is a bounding cochain.
Then $\bb_0$ can be regarded as a choice of branch associated to $\rho$ 
which affects the curvature term of the $A_\infty$ structure only.
Therefore we may regard $\rho$ instead of $\bb_0$ as part of the relevant data 
consisting of an object of a Fukaya category.

\begin{definition}\label{lagbrane}
We call the triple $\lag = (L,\rho,\bb_+)$ an equivariant Lagrangian brane, 
where $\rho$ is a $\C$-local system of rank $1$ on $L$ such that there exists 
a closed $1$-form $\bb_0 \in \Omega^1(L;\C)$ satisfying $\rho(\gamma)=e^{\langle\bb_0,\gamma\rangle}$ 
for all $\gamma\in\pi_1(L)$ and $\bb:=\bb_0+\bb_+$ is a bounding cochain 
in the sense of Definition \ref{boundingc}.
\end{definition}
As we only use equivariant Lagrangian branes, 
we sometimes call them Lagrangian branes for short hereafter.
We also call $L$ the underlying Lagrangian submanifold of $\lag$.

%%% Equivariant Fukaya category
\subsection{Category of equivariant Lagrangian branes}\label{fukaya}

Let $G$ be a compact connected Lie group 
and $\g := Lie(G)\otimes_\R\C$ be its complexified Lie algebra, 
$M$ be a compact $G$-manifold equipped with a $G$-invariant symplectic form $\omega$ 
and a $G$-invariant $\omega$-compatible almost complex structure.
In this section we first introduce 
the $\g$-differential gapped filtered $A_\infty$ category 
of equivariant Lagrangian branes and 
then deform it with bounding cochains 
to get a Fukaya $A_\infty$ category over $\LL$.

Take a finite collection of pairs of 
compact oriented spin Lagrangian submanifolds preserved by the $G$-action 
and a $\C$-local system of rank $1$ on it 
$\Lag = \{ (L,\rho) \}$, 
such that 
(i) each $L$ satisfies Assumption \ref{regularity}, 
(ii) any pair $L$ and $L'$ either coincide or do not intersect, 
and (iii) each $\rho$ can be expressed as 
$\rho(\gamma)=e^{\langle \bb_{0},\gamma \rangle}$ for some closed $1$-form 
$\bb_{0}\in\Omega^1(L;\C)$.
We first construct a unital $\g$-differential gapped filtered $A_\infty$ category 
$\F_\Lag$ with the set of objects $\Lag$ as follows.

Set 
\begin{align*}
\cF_\Lag((L,\rho),(L',\rho')) &:= 
\left\{ \begin{array}{lc}
\Omega^\ast(L;\hhom(\rho,\rho')) & {\rm if\ } L=L',\\
0 & {\rm otherwise}
\end{array} \right.
\end{align*}
Then this $\cF_\Lag((L,\rho),(L',\rho'))$ is a $\g$-differential space 
with the operators $L$ and $i$.

Take a sequence of $(k+1)$ objects $(L_0,\rho_0),\ldots,(L_k,\rho_k)$ 
such that $L_0=\cdots=L_k=:L$.
Recall the evaluation maps from the compactified moduli space of 
holomorphic disks 
\begin{align}
\xymatrix{
L^k && \cM_{k,\beta}(L) \ar[ll]_-{\ev_{1}\times\cdots\times\ev_{k}}
\ar[r]^-{\ev_0} & L.
} \nonumber
\end{align}
Denote $\rho_{j-1,j} := 
\hhom(\rho_{j-1},\rho_{j})$ 
and $\rho_{0,k} := 
\hhom(\rho_0,\rho_k)$.
Consider $1$-parameter families of evaluation maps 
$\ev_{j,t}$ along the $j$-th boundary arc $\partial_j$ 
connecting the $j$-th marked point to the $(j+1)$-th marked point for $0\leq j\leq k$ 
and $\partial_k$ connecting the $k$-th marked point to the $0$-th marked point, 
and define $P_j$ to be the bundle isomorphism 
$\ev^\ast_j\rho_{j'-1,j'}\to\ev^\ast_{j+1}\rho_{j'-1,j'}$ between the pull back of local systems 
along $\ev_{j,t}$ obtained by parallel transport 
(where $\ev_{k+1} := \ev_0$).
$P_j$'s are 
independent of the choices of $\ev_{j,t}$, and 
we define
\begin{align*}
\m_{k,\beta}(x_1,\ldots,x_k) &:=
(-1)^\#\ev_{0\ast}(\ev_1^\ast(\alpha_1)\wedge\cdots\wedge\ev_k^\ast(\alpha_k) \\
& \otimes P_k(s_k)\circ (P_k\circ P_{k-1}) (s_{k-1}) 
\circ\cdots\circ (P_k\circ \cdots\circ P_1)(s_1) ) \\
\intertext{for $\beta\neq 0$ or $k\geq 2$ where $\# = \sum_{i=1}^k i|x_i|' +1$, and}
\m_{1,0}(x_1) &:= dx_1 
\end{align*}
for $x_1=\alpha_1\otimes s_1\in \Omega^\ast(L;\hhom(\rho_{0},\rho_{1})) , \dots , 
x_k = \alpha_k\otimes s_k \in \Omega^\ast(L;\hhom(\rho_{k-1},\rho_{k}))$.

It is easy to see that these $\m_{k,\beta}$'s satisfy $A_\infty$ relations 
by carrying out the proof of Theorem \ref{fukayacat} with local systems.
They also satisfy the compatibility with the interior product,
\begin{align}
i_X \m_{k,\beta} (x_1,\ldots,x_k) + \sum_{i=1}^k (-1)^{|x_1|'+\cdots+|x_{i-1}|'} 
\m_{k,\beta}(x_1,\ldots,i_X x_i,\ldots,x_k) = 0 . \nonumber
\end{align}
These operators satisfy gapping conditions and the quadruple 
$(\cF_\Lag,\{\m_{k,\beta}\},L,i)$ forms 
a $\g$-differential gapped filtered $A_\infty$ category.

Next, take a finite set of equivariant Lagrangian branes 
$\{ \lag_{\alpha} = (L,\rho_\alpha,\bb_{+,\alpha}) \}$ 
with $(L,\rho_\alpha)\in\Lag$ 
and choose for each $\lag_{\alpha}$ 
a closed $1$-form $\bb_{0,\alpha} \in \Omega^1(L;\C)$ of the local system $\rho_\alpha$ 
such that 
$\rho_\alpha(\gamma)=e^{\langle \bb_{0,\alpha},\gamma \rangle}$ 
and that $\bb_\alpha := \bb_{0,\alpha}+\bb_{+,\alpha}$ is a bounding cochain.
We denote by $\lag$ the set of such pairs $(L,\bb_\alpha)$.

We're going to construct an uncurved $A_\infty$ category $\F_\lag$ 
with the set of objects $\lag$.
For each $(L,\bb_\alpha)$ we define its curvature $c_\alpha\in\Lambda_+$ by 
$\m_0^{\bb,\lambda}(1)=c_\alpha\cdot 1_{\lag_\alpha}$.

\begin{definition}\label{F}
$\F_\lag$ is 
a $\Z/2\Z$-graded uncurved $A_\infty$ category over $\Lambda$ 
with 
$\ob\F_\lag = \lag$, 
\begin{align}
\F_\lag((L,\bb_\alpha),(L',\bb_{\alpha'})) &:= 
\left\{ \begin{array}{cl}
\cF_\Lag((L,\rho_\alpha),(L',\rho_{\alpha'}))^\g\hat\uotimes\Lambda&
{\rm if \ }L=L',i_X(\bb_\alpha)=i_X(\bb_{\alpha'})\\
&{\rm for \ all \ }X\in\g {\rm \ and \ }
c_\alpha=c_{\alpha'}\\
0 & {\rm otherwise}
\end{array}\right. \nonumber
\end{align}
with operators $\m_0 := 0$ and
\begin{align}
\m_k(x_1,\ldots,x_k) &:= 
\m_k^\bb(x_1,\ldots,x_k) - \left\{ \begin{array}{lr}
0 & k \neq 0,1\\
\sum_{j=1}^r\lambda(F^j)i_{e_j}(x_1) & k = 1
\end{array} \right. \nonumber
\end{align}
for $k \geq 1$, where
\begin{align}
\m_k^{\bb}(x_1,...,x_k) :=& 
\sum_\beta T^{\omega(\beta)/2\pi}
\sum_{l_0+\cdots+l_k=l}
\m_{k+l,\beta}
(\overbrace{\bb_{+,0},\ldots,\bb_{+,0}}^{l_0},x_1,
\ldots,x_k,
\overbrace{\bb_{+,k},...,\bb_{+,k}}^{l_k}) \nonumber
\end{align}
whenever the spaces of morphisms concerned are nonzero.
\end{definition}
Note that the $\m_{k,\beta}$'s appeared in the above are the operators 
defined in this subsection.
It is easy to check that $\m_k$'s satisfy the $A_\infty$ relations.

%%% Matrix factorization %%%
\section{Matrix factorizations}\label{mf}

\subsection{Preliminaries on categories}\label{categories}
In \S \ref{categories} we recall some basic definitions on derived categories.
See \cite[\S5]{dyccom} for more details.

Let $k$ be an algebraically closed field of characteristic zero.
For a trianglurated category $\mathcal{T}$,
the idempotent completion of $\mathcal{T}$ is denoted by $\overline{\mathcal{T}}$ 
and $\mathcal{T}$ is said to be idempotent complete if $\mathcal{T}$ is naturally equivalent to $\overline{\mathcal{T}}$ 
(see, e.g., \cite{balide}).

In this section, a differential $\Z/2\Z$-graded category over $k$ is briefly called a dg-category. 
Let $T$ be a dg-category.
The $k$-linear category $[T]$ is defined by taking even cohomology $H^0$ of $T$. 
Let $T^\op$-$\module$ be the $k$-linear category of right $T$-modules, i.e., a category of dg functors from $T^\op$ to the dg-category of $\Z/2\Z$-graded complexes over $k$.
We denote by $D(T^\op)$ the localization of $T^\op$-$\module$ with respect to the set of objectwise quasi-isomorphisms.
Then $D(T^\op)$ is an idempotent complete triangulated category which admits arbitrary coproducts.
By the Yoneda embedding, $[T]$ is considered as a full subcategory of $D(T^\op)$. 
Let $[\widehat{T}_{\pe}]$ be the full subcategory of compact objects in $D(T^\op).$
Then $[\widehat{T}_{\pe}]$ is the smallest triangulated subcategory of $D(T^\op)$ containing $[T]$ and closed under direct summands.
%%%%%%%%%%%%%%%%%%%%%%%%%%%                REMARK ON NOTATION                 %%%%%%%%%%%%%%%%%%%%%%%%%%%%%%%%%%%%%
\begin{remark}\label{label}
For a dg-category $T,$ an associated $\Z/2\Z$-graded $A_\infty$ category $T_\infty$ is defined as follows:\\
The set of objects of $T_\infty$ is the same as $T$.
For objects $X$ and $Y$, the morphism space $T_\infty(X, Y)$ is $T(Y, X)$ where $T(Y, X)$ is the morphism space of the dg-category $T.$
The $A_\infty$ structure $\{\m_k\}$ are defined by
 \[\m_1(x_1):=dx_1, \ \m_2(x_1, x_2):=(-1)^{|x_1|}x_1 \cdot x_2,\ \m_k=0 \ (k \ge 3).\]
\end{remark}
%%%%%%%%%%%%%%%%%%%%%%%%%%%%%%%%%                      Definition of matrix factorization             %%%%%%%%%%%%%%%%%%%%%%
\subsection{Preliminaries on matrix factorizations}
Let $R$ be a commutative regular $k$-algebra with finite Krull dimension $n$. 
Take $w \in R \setminus k$.
We define a matrix factorization of $w$ by the pair $(P, d_P)$, 
where $P=P^0 \oplus P^1$ is a $\Z/2\Z$-graded finitely generated projective $R$-module 
and $d_P \in \mathrm{End}^{\mathrm{odd}}(P)$ is an $R$-linear morphism of odd degree with $d_P^2=w \cdot \mathrm{id}_P$.
Then $d_P$ consists of $\varphi \in \Hom_R(P^1, P^0)$ 
and $\psi \in \Hom_R(P^0, P^1)$ with 
$\varphi \circ \psi=w \cdot \mathrm{id}_{P_0}$, $\psi \circ \varphi = w \cdot \mathrm{id}_{P_1}$.
For matrix factorizations $(P, d_P)$ and $(P', d_{P'})$,  
 the $\Z/2\Z$-graded module of $R-$ linear morphisms from $P$ to $P'$ 
 with a differential 
 \[d(f):=d_{P'} \circ f-(-1)^{|f|} f \circ d_{P}\] is denoted by  $\mf(P, P')$.
 Here $f$ is a homogeneous $R$-linear morphism and $|f|\in \Z/ 2\Z$ is the degree of $f$.
 These data define a dg-category $\mf(w)$, where compositions of morphisms are naturally defined.
 Then $[\mf(w)]$ is a triangulated category.
 
%%%%%%%%%%%%%%%%%%%%%%%%             Categories of singularities               %%%%%%%%%%%%%%%%%%%%%%%%%%%% 
Set $S:=R/w$.
Let $D^b(S)$ be the derived category of complexes of $S$-modules with finitely generated total cohomology.
A complex of $S$-modules is called perfect if it is quasi-isomorphic to a bounded complex of finitely generated projective $S$-modules. 
We denote by $D^b_\perf(S)$ the subcategory of perfect complexes in $D^b(S)$.
Then $D^b_\perf(S)$ is a thick subcategory of $D^b(S)$ and the Verdier quotient $D^b(S)/D^b_\perf(S)$ is denoted by $\underline{D}^b(S)$,
which is called a stabilized derived category of $S$.
In some references, a stabilized derived category is also called a triangulated category of singularity.
There exists a triangulated functor 
\begin{align}\label{cok}
\mathrm{cok} : [\mf(w)] \to \underline{D}^b(S),
\end{align}
which sends a matrix factorization $(P, d_P)$ to $\mathrm{cok}(\varphi)$.
Moreover, this functor gives an equivalence of triangulated categories. 

%%%%%%%%%%%%%%%%%%%%%%%%           Stabilization                  %%%%%%%%%%%%%%%%%%%%%%%%%%%%%%%%%%%%%%%%
Let $\mathrm{L}$ be a finitely generated $S$-module.
Then $\mathrm{L}$ is naturally considered as an object of $\underline{D}^b(S)$.
A matrix factorization  $\mathrm{L}^\stab$ is defined by $\mathrm{cok}(\mathrm{L}^\stab)=\mathrm{L}$,
which is called a stabilization of $\mathrm{L}$.

Let $f_1, f_2, \dots, f_m$ be a regular sequence in $R$ and let $I \neq R$ be the ideal generated by $f_1, f_2, \dots, f_m$.
We assume that $I$ contains $w$.
Take $w_1, w_2, \dots, w_m \in R$ with $w=w_1f_1+w_2f_2+\cdots + f_m w_m$. 
Let $V \cong R^m$ be the free $R$-module of rank $m$ with a basis $e_1, e_2, \dots, e_m$
and let $e_1^\vee, e_2^\vee, \dots e_m^\vee$ be the dual basis.
The contraction by $e_i^\vee$ is denoted by $\iota_i \in \mathrm{End}\big(\displaystyle{\bigwedge^*} V\big)$.
We define $s_0, s_1\in \mathrm{End}\big(\displaystyle{\bigwedge^*} V\big)$ by 
\begin{align*}
s_0&:=f_1\iota_1+f_2\iota_2+\cdots+f_m \iota_m\\
s_1&:=w_1e_1\wedge+w_2e_2\wedge+\cdots+w_me_m\wedge.
\end{align*}
Then we easily see that the $\Z/2\Z$-graded $R$-module $\displaystyle{\bigwedge^*}V$ equipped with the odd degree morphism $s_0+s_1$ is a matrix factorization of $w$.
By \cite[COROLLARY 2.7]{dyccom}, this matrix factorization is a stabilization of the $S$-module $R/I$.
If $R$ is a local ring with the maximal ideal $\m$ and  $I=\m$, then this stabilization is denoted by $k^\stab$.

Let $\crit(w)$ be the critical locus of $w$, i.e., the scheme-theoretic zero locus of $dw$.
Set \[\sing(S):=\crit(w) \cap \Spec(S),\]
which is the singular locus of $\Spec(S)$.
If $\crit(w)$ (resp. $\sing(S)$) is zero-dimensional, then we say that  $w$ (resp. $S$) has isolated singularities.

Let $\m$ be a maximal ideal of $R$ and
let $\widehat{R}_\m$ be the completion of the local ring $R_\m$with respect to the $\m$-adic topology.
The element of $\widehat{R}_\m$ corresponding to $w$ is denoted by $\hat{w}_\m$.
By \cite[THEOREM 5.7]{dyccom}, it follows that $[\mf(\hat{w}_\m)]$ is idempotent complete
if $w \in \m$ and $\hat{w}_\m$ has isolated singularities. 
There exists a restriction functor from $\mf(w)$ to $\mf(\hat{w}_\m)$.
By the equivalence (\ref{cok}) and \cite[Theorem 2.10]{orlfor} (see also \cite[Proposition 3.4]{orlfor}),
we see that these restriction functors give an equivalence between triangulated categories
\[\overline{[\mf({w})]} \cong \prod_{\m \in \sing(S)} [\mf(\hat{w}_\m)]\]
if $S$ has isolated singularities.
By \cite[THEOREM 5.2, COROLLARY 5.3 and THEOREM 5.7]{dyccom}, we have the following:
\begin{theorem}
Suppose that $(R, \m)$ is a local $k$-algebra and $w \in \m$ has isolated singularities.
Then $k^\stab$ split-generates $[\mf(\hat{w}_\m)]$, i.e.,
the smallest triangulated subcategory of $[\mf(\hat{w}_\m)]$ 
containing $k^\stab$ and closed under direct summand is $[\mf(\hat{w}_m)]$.
Set $A:=\mf(k^\stab, k^\stab)$, which is considered as a dg-category with one object.
Then, the Yoneda embedding gives an equivalence between triangulated categories 
\[[\mf(\hat{w}_\m)] \cong [\widehat{A}_\pe].\]
\end{theorem}
Let $W:=\Hom_k(\m/\m^2, k)$ be the Zariski tangent space of $(R, \m)$.
A Hessian matrix of $w$ gives a quadratic form on $W$ and let $\cl(w)$ be the corresponding Clifford algebra.
If $\m$ is a non-degenerate critical point of $w$, i.e., Hessian is non-degenerate, then 
we have a natural inclusion $\cl(-w) \subseteq \mf(k^\stab, k^\stab)$ which gives a quasi-isomorphism (see, e.g., \cite[\S 5.5]{dyccom}). 

%%%%%%%%%%%%%%%%%%%%%%%%%%%%               CATEGOROES OF BRANE                    %%%%%%%%%%%%%%%%%%%%%%%%%%%%%%%
We define a dg-category of branes $\br(w)$ by 
\[\br(w):=\prod_{c \in k}\mf(w-c),\]
then $[\br(w)]$ is a triangulated category.
Moreover, if $w$ has isolated singularities, then we have  
\[[\overline{\br(w)}] \cong \prod_{\m \in \crit(w)}[\mf(\hat{w}_\m)].\]
%%%%%%%%%%%%%%%%%%%%%%%%%%%%               MF FOR GIVENTAL POTENTIAL              %%%%%%%%%%%%%%%%%%%%%%%%
\subsection{Categories of branes for Givental type potential functions}\label{givental}
Let $R$ be a commutative regular $k$-algebra with finite Krull dimension $n$. 
Choose $f,g_1,...,g_r \in R$ and $\lambda_1, \lambda_2, \dots, \lambda_r \in k$.
Assume that $g_1,...,g_r$ are invertible.
For abbreviation, these data is denoted by 
$f - \lambda_1 \log g_1 - \cdots - \lambda_r \log g_r$ or simply by $F$, which is called a Givental type potential function.
Although $F$ is not well-defined as a single-valued function, 
its differential can be defined as:
\begin{align}
dF=df-\lambda_1\dfrac{dg_1}{g_1} - \cdots - \lambda_r \dfrac{dg_r}{g_r}. \nonumber
\end{align}
We define $\crit (F)$ as the scheme theoretic zero-locus of $dF$. 
We assume that $F$ has isolated singularities, i.e., $\crit (F)$ is zero-dimensional. 

Let $\m \in \Spm(R)$ be a maximal ideal of $R$ and $(\widehat{R}_\m, \hat{\m})$ be the completion with respect to the $\m$-adic topology,
where $\hat{\m}$ is the maximal ideal. 
For $h \in R$, the corresponding element in $\widehat{R}_\m$ is denoted by $\hat{h}_\m$.
For an invertible element $h \in \widehat{R}_\m\setminus \hat{\m}$, we define $\log h \in \hat{m}$ by 
\[\log h =\sum_{k=1}^\infty (-1)^{k-1} \dfrac{1}{k} \left( \dfrac{h-h(0)}{h(0)} \right)^k,\]
where $h(0) \neq 0 \in k$ is the value of $h$ at $\hat{\m}$. 
Then $\hat{F}_\m$ at $\m \in \Spm(R)$ 
can be defined as 
\begin{align}
\hat{F}_\m := \hat{f}_\m-\hat{f}_\m(0) - \lambda_1 \log \hat{g}_{1,\m} - \cdots - \lambda_r \log \hat{g}_{r, \m} 
\in \hat{\m}. \nonumber
\end{align}

\begin{definition}
The dg-category of matrix factorizations 
of the Givental type potential $F$ is defined to be
\begin{align}
\br (F) &:= \prod_{\m \in \crit (F)}\mf (\hat{F}_\m). \nonumber
\end{align}
The idempotent complete triangulated category $[\br(F)]$ is also called the category of matrix factorizations of $F.$
\end{definition}
For $\m \in \crit(F),$ the matrix factoriziation $(R/\m)^\stab \in \br(F)$ (or $[\br(F)]$) is simply denoted by $k_\m^\stab$.
Then these matrix factorizations split generate $[\br(F)].$  

%%% Homological mirror symmetry %%%
\section{Equivariant homological mirror symmetry}\label{equivhms}

In this section we study the cases $M = \C P^1$ and $\C$.
As explained in the introduction, 
we consider the equivariant Floer $A_\infty$ algebras of 
moment fiber Lagrangians of $M$, following 
the study in the non-equivariant case by 
Fukaya-Oh-Ohta-Ono \cite{fooo:toric1}.

\subsection{The case of $\C P^1$}

Let $\C P^1$ be equipped with the $S^1=U(1)$-action 
$[z_0:z_1] \mapsto [z_0:\zeta z_1]$ ($\zeta \in U(1)$), 
an $S^1$-invariant symplectic form $\omega$ with $\omega(\C P^1) = 2\pi$.
Let $\mu \colon \C P^1 \to [0,1]$ be the associated moment map 
and let $L_u := \mu^{-1}(u) \subset \C P^1$ be 
the moment fiber Lagrangian over an interior point $u \in (0,1)$.
By identifying $L_u$ with $S^1$ via the $S^1$-action, 
we choose an orientation of $L_u$ 
such that it is compatible with that of $S^1$.
We choose the standard spin structure of $L_u$ (\cite[Section 8]{MR2057871}).

Let $\g := Lie(S^1)\otimes\C$.
Take a basis $e_1$ of $\g$ consistent with the orientation of $S^1$ and 
an integral basis $\mathbf{e}^1 \in \Omega^1(L_u)^\g \cong H^1(L_u)$, 
such that $\mathbf{e}^1$ coincides with the dual basis $e^1 \in \g^\vee$ 
via the identification $S^1 \to L_u$ given by the $S^1$-action.

As we saw in Remark \ref{maslov2}, 
only constant disks and holomorphic disks 
with Maslov index $2$ contribute the equivariant Floer $A_\infty$ algebras of $L_u$.
There are two holomorphic disks of Maslov index $2$ 
up to automorphisms: 
the one which projects onto $[0,u] \subset [0,1]$ via $\mu$ 
and the other one which projects onto $[u,1] \subset [0,1]$, see \cite{CO}.
We denote their relative homology classes by 
$\beta_1, \beta_2 \in H_2(\C P^1,L_u;\Z)$ respectively.
Then $\langle \mathbf{e}^1,\partial\beta_1\rangle =1$, 
$\omega(\beta_1) = 2\pi u, \, \omega(\beta_2) = 2\pi (1-u)$ 
and $\partial \beta_2 = - \partial \beta_1$.

Take a closed $S^1$-invariant $1$-form $\bb = \bb_0 + \bb_+$ where $\bb_0 \in \Omega^1(L_u)$, 
$\bb_+ \in \Omega^1(L_u) \hat\uotimes \Lambda_+$ and put 
\begin{align}
c_0 := e^{\langle \bb_0,\partial\beta_1 \rangle} \in \C, \qquad
1+c_+ := e^{\langle \bb_+,\partial\beta_1 \rangle} \in \LLp. \label{c}
\end{align}

Take $\lambda \in \Lambda_0$ and take an equivariant parameter 
$\lambda \colon S\g^\vee \cong \C[e^1] \to \Lambda_0$ 
which sends $e^1$ to $\lambda$.
Recall that the curvature term of the equivariant Floer $A_\infty$ algebra 
evaluated at $\lambda$ and 
deformed by $\bb$ is given by
\begin{align}
\m_0^{\bb,\lambda}(1) &= T^{\omega(\beta_1)/2\pi} e^{\langle \bb,\partial \beta_1 \rangle} 
+ T^{\omega(\beta_2)/2\pi} e^{\langle \bb,\partial \beta_2 \rangle} 
- \lambda i_{e_1}(\bb) \nonumber\\
&=T^u c_0(1+c_+) + T^{1-u} c_0^{-1}(1+c_+)^{-1} 
- \lambda \langle \bb,e_1 \rangle . \nonumber\\
\intertext{The differential is calculated as}
\m_1^{\bb,\lambda}(\mathbf{e}^1) &= 
\sum_\beta T^{\omega(\beta)/2\pi} e^{\langle \bb_0,\partial \beta \rangle}
\sum_{l_0+l_1=l} 
\m_{l+1,\beta}(\overbrace{\bb_+,...,\bb_+}^{l_0},e^1,\overbrace{\bb_+,...,\bb_+}^{l_1})
- \lambda i_{e_1}(\mathbf{e}^1) \nonumber\\
&=T^uc_0(1+c_+) - T^{1-u}c_0^{-1}(1+c_+)^{-1} - \lambda \nonumber
\end{align}
by using the divisor axiom.
Givental's potential function for $\C P^1$
\begin{align}
F &= T^u x + T^{1-u}x^{-1} - \lambda \log x 
= X + \dfrac{T}{X} - \lambda \log X + {\rm const.} \nonumber
\end{align}
is recovered from the curvature term $\m_0^{\bb,\lambda}(1)$ 
by introducing the variables $x := c_0(1+c_+)$ and $X := T^u x$.
The condition that the differential $\m_1^{\bb,\lambda}$ vanishes 
coincides with the equation for the critical points
\begin{align}
\partial F = X-T/X-\lambda = 0, \label{criticalpts}
\end{align}
where $\displaystyle \partial := X \frac{\partial}{\partial X}$.

In the following we assume $\lambda \neq \pm 2\sqrt{-1}T^\frac{1}{2}$ 
so that (\ref{criticalpts}) does not have a double root, 
which implies the equivariant Landau-Ginzburg mirror potential $F$ 
does not have a degenerate critical point.
The case $\lambda = \pm 2\sqrt{-1}T^\frac{1}{2}$ can be treated similarly, 
see Remark \ref{degenerate}.

We associate an equivariant Lagrangian brane to a solution $X$ of (\ref{criticalpts}) 
satisfying $0<\val(X)<1$ as follows.
Take the moment fiber $L := L_u$ where $u := \val (X)$, 
take $c_0 \in \C^\ast$ and $c_+ \in \Lambda_+$ such that 
$X = T^uc_0(1+c_+)$.
The positive valuation part of the bounding cochain $\bb_+$ is given by 
$\displaystyle \bb_+ = \log (1+c_+)\mathbf{e}^1$ 
where $\log (1+c_+) := \sum_{k\geq 1}(-1)^{k-1}\frac{c_+^k}{k}$.
The leading term $\bb_0 = b_0\mathbf{e}^1$ is taken such that $e^{b_0} = c_0$, 
and gives a local system $\rho$ on $L_u$ with the monodromy $c_0$ 
($\rho$ is independent of the choice of $\bb_0$).
Then $\bb = \bb_0+\bb_+$ gives a bounding cochain 
since $\mathbf{e}^1$ is an invariant form, 
and we obtain an equivariant Lagrangian brane $(L_u,\rho,\bb_+)$.
This construction gives a one-to-one correspondence 
between the solutions of (\ref{criticalpts}) with $\val(X)\in(0,1)$ and 
Lagrangian branes (with standard spin structures) 
whose equivariant Floer $A_\infty$ algebra is not 
quasi-isomorphic to $0$.
This is compatible with (\ref{c}).

Next we study the solutions of (\ref{criticalpts}).
It has two solutions $X_1,X_2$. 
Then
\begin{enumerate}
\item if $\val (\lambda) \geq \dfrac{1}{2}$, 
both solutions have valuation $u_1=u_2=\dfrac{1}{2}$, i.e.\ 
we consider two Lagrangian branes 
with {\it the same} underlying Lagrangian submanifold $L_{\frac{1}{2}}$ 
but with different bounding cochains, and 
\item if $0 \leq \val (\lambda) < \dfrac{1}{2}$, 
the solutions have valuations $u_1=\val (\lambda)$ and $u_2=1-\val(\lambda)$ 
respectively 
and therefore we have two disjoint underlying Lagrangians 
$L_{u_1}$ and $L_{u_2}$.
We only consider the case $\val(\lambda)>0$ 
because the Lagrangians collapse to points when $\val(\lambda)=0$.
\end{enumerate}
Let us denote the corresponding maximal ideals by $\m_1,\m_2 \in \crit F$ 
and the corresponding bounding cochains by 
$\bb_i=\bb_{0,i}+\bb_{+,i}$ ($i=1,2$).
Since $\m_1^{\bb_i,\lambda}=0$, 
\[
H(\Omega^\ast(L_{u_i})^\g\hat\uotimes\Lambda,\m_1^{\bb_i,\lambda}) 
\cong H^\ast(L_{u_i})\underset{\C}{\otimes}\Lambda
\]
as $\Z/2\Z$-graded $\Lambda$-vector spaces.

Next we see the multiplicative structure.
$\m_2^{\bb_i,\lambda}$ of the equivariant Floer $A_\infty$ algebra 
is calculated as:
\begin{align}
2\m_2^{\bb_i,\lambda}(\mathbf{e}^1,\mathbf{e}^1) 
&= \m_2^{\bb_i,\lambda}(\mathbf{e}^1,\mathbf{e}^1)
+\m_2^{\bb_i,\lambda}(\mathbf{e}^1,\mathbf{e}^1)\nonumber\\
&=\sum_\beta T^{\omega(\beta)/2\pi}e^{\langle\bb_{0,i},\partial\beta\rangle}
\sum_{l_0+l_1+l_2=l}
(
\m_{l+2,\beta}(\overbrace{\bb_{+,i},...,\bb_{+,i}}^{l_0},\mathbf{e}^1,
\overbrace{\bb_{+,i},...,\bb_{+,i}}^{l_1},\mathbf{e}^1,
\overbrace{\bb_{+,i},...,\bb_{+,i}}^{l_2})\nonumber\\
&\qquad +\m_{l+2,\beta}(\overbrace{\bb_{+,i},...,\bb_{+,i}}^{l_0},\mathbf{e}^1,
\overbrace{\bb_{+,i},...,\bb_{+,i}}^{l_1},\mathbf{e}^1,
\overbrace{\bb_{+,i},...,\bb_{+,i}}^{l_2}))\nonumber
\\
&=\sum_\beta T^{\omega(\beta)/2\pi}e^{\langle\bb_{0,i},\partial\beta\rangle}
\langle\bb_{+,i},\partial\beta\rangle
\sum_{l_0+l_1=l}
\m_{l+1,\beta}(\overbrace{\bb_{+,i},...,\bb_{+,i}}^{l_0},\mathbf{e}^1,
\overbrace{\bb_{+,i},...,\bb_{+,i}}^{l_1})\nonumber\\
&=\sum_\beta T^{\omega(\beta)/2\pi}e^{\langle\bb_{0,i},\partial\beta\rangle}
\langle\bb_{+,i},\partial\beta\rangle^2\sum_l\m_l(
\bb_{+,i},...,\bb_{+,i}
)\nonumber\\
&=T^uc_0(1+c_+)+T^{1-u}c_0^{-1}(1+c_+)^{-1}\nonumber\\
&=X+TX^{-1} \nonumber
\end{align}
by applying the divisor axiom twice.
Hence we have $2\m_2^{\bb_i,\lambda}(\mathbf{e}^1,\mathbf{e}^1) = \partial^2 F(\m_i)$.
This is nonzero since we assumed $\lambda\neq\pm2\sqrt{-1}T^{\frac{1}{2}}$.

Let $H(\Omega^\ast(L_{u_i})^\g\hat\uotimes\Lambda,\m_1^{\bb_i,\lambda})$ 
be the cohomology algebra associated with 
the equivariant Floer $A_\infty$ algebra 
$(\Omega^\ast(L_{u_i})^\g\hat\uotimes\Lambda,\{\m_k^{\bb_i,\lambda}\})$, 
with the product defined as 
\[
[x]\cdot[y] := (-1)^{|x|}[\m_2^{\bb_i,\lambda}(x,y)].
\]
Then we have an isomorphism of $\Lambda$-algebras 
\begin{align}
H(\Omega^\ast(L_{u_i})^\g\hat\uotimes\Lambda,\m_1^{\bb_i,\lambda}) 
\to \cl (-\hat{F}_{\m_i}) \nonumber
\end{align}
by sending $[\mathbf{e}^1]$ to the generator of the Clifford algebra.
By Sheridan's intrinsic formality for the Clifford algebras \cite[Corollary 6.4]{sheridan}, 
we can lift it to a quasi-isomorphism of 
$A_\infty$ algebras
\begin{align}
(\Omega^\ast(L_{u_i})^\g\hat\uotimes\Lambda,\{\m_k^{\bb_i,\lambda}\}) 
\to \cl (-\hat{F}_{\m_i})_\infty, \label{floercl}
\end{align}
where $\cl (-\hat{F}_{\m_i})_\infty$ is defined in Remark \ref{label}.

Now let $\lag$ be the set of pairs $\{(L_{u_1},\bb_1),(L_{u_2},\bb_2)\}$ 
and let $\F_\lag$ be the uncurved $A_\infty$ category over $\Lambda$ 
as defined in Definition \ref{F}.
Since $i_{e_1}(\bb_1)\neq i_{e_1}(\bb_2)$, 
we have $\F_\lag((L_{u_i},\bb_i),(L_{u_j},\bb_j))=0$ if $i\neq j$.
Then by (\ref{floercl}) and by Dyckerhoff's result reviewed in Section 4.3, 
we have an objectwise $A_\infty$ functor
\begin{align}
\F_\lag((L_{u_i},\bb_i),(L_{u_i},\bb_i)) \stackrel{\sim}\longrightarrow \cl(-\hat{F}_{\m_i})_\infty 
\subseteq 
\mf (\Lambda^\stab_{\m_i},\Lambda^\stab_{\m_i})_\infty , \nonumber
\end{align}
where $\mf$ denotes the dg-category of matrix factorizations of $F$.
Therefore we have the following.
\begin{theorem}
[Equivariant homological mirror symmetry for $\C P^1$]\label{maintheorem1}

We have a cohomologically fully-faithful $A_\infty$ functor 
for $\val(\lambda)>0$ and $\lambda\neq\pm2\sqrt{-1}T^\frac{1}{2}$
\begin{align*}
\phi \colon \F_\lag \to \br(F)_\infty
\end{align*}
by sending $(L_{u_i},\bb_i)$ to $\Lambda^\stab_{\m_i}$, 
whose image split-generates $[\br(F)]$.
\end{theorem}
\begin{remark}[Degenerate cases]\label{degenerate}
Suppose that $\lambda=\pm2\sqrt{-1}T^\frac{1}{2}$.
Then the potential function $F$ has a unique degenerate critical point $\m$ with valuation $\frac{1}{2}$.
Let $(L_\frac{1}{2}, \rho, \bb=\bb_0+\bb_+)$ be the equivariant Lagrangian brane with a bounding cochain corresponding to $\m$.
By using the divisor axiom, we easily see that
\begin{align}\label{higher}
k!\cdot\m^{\bb, \lambda}_k(\mathbf{e}^1, \mathbf{e}^1, \dots, \mathbf{e}^1)=\partial^kF(\m).
\end{align}
We note that all structure constants of the equivariant Floer algebra are determined by the above equation {\rm (\ref{higher})}.
On the other hand, by \cite[Theorem 5.8]{dyccom}, this unital $A_\infty$ algebra is quasi-isomorphic to the 
$A_\infty$ algebra $\mf(\Lambda^\stab_\m, \Lambda^\stab_\m)_\infty$.
\end{remark}
This means that $\F_\lag$ and $\br(F)_\infty$ are Morita equivalent.

%%% C %%%
\subsection{The case of $\C$}

Let $\C$ be equipped with the $S^1=U(1)$-action 
$z \mapsto \zeta z$ ($\zeta \in U(1)$), 
the standard symplectic form $\omega = dx \wedge dy$ 
($z = x+iy$) and the standard complex structure. 
Let $\mu \colon \C \to \R_{\geq 0}; \, z \mapsto \frac{1}{2}|z|^2$ be 
the moment map associated to the $S^1$-action, 
and let $L_u := \mu^{-1}(u)$ be the moment fiber Lagrangian over $u \in \R_{>0}$.
By identifying $L_u$ with $S^1$ via the $S^1$-action, 
we choose an orientation of $L_u$ 
such that it is compatible with that of $S^1$.
We choose the standard spin structure of $L_u$ (\cite[Section 8]{MR2057871}).

Let $\g := Lie(S^1)\otimes\C$.
Take a basis $e_1$ of $\g$ consistent with the orientation of $S^1$ and 
an integral basis $\mathbf{e}^1 \in \Omega^1(L_u)^\g \cong H^1(L_u)$, 
such that $\mathbf{e}^1$ coincides with the dual basis $e^1 \in \g^\vee$ 
via the identification $S^1 \to L_u$ given by the $S^1$-action.

Although $\C$ is noncompact, the arguments in Section \ref{equivfloer} 
applies to this case as well, 
because all the holomorphic and stable disks bounded by $L_u$ are 
contained in a compact set thanks to the maximum principle.
There is a unique holomorphic disk up to automorphisms 
of Maslov index $2$ bounded by $L_u$, 
which projects onto $[0,u] \subset \R_{\geq 0}$ and whose relative homology class 
we denote by $\beta \in H_2(\C P^1,L_u;\Z)$.
Then $\langle \mathbf{e}^1,\partial \beta \rangle = 1$ and $\omega(\beta) = 2\pi u$.

Take a closed $S^1$-invariant $1$-form 
$\bb = \bb_0+\bb_+$ where $\bb_0 \in \Omega^1(L_u)$ and 
$\bb_+ \in \Omega^1(L_u) \hat\uotimes \LLp$, 
and put 
\begin{align}
c_0 := e^{\langle \bb_0,\partial\beta \rangle} \in \C, \qquad
1+c_+ := e^{\langle \bb_+,\partial\beta \rangle} \in \LLp. \label{c}
\end{align}
Take $\lambda \in \Lambda_0$ and take an equivariant parameter 
$\lambda \colon S\g^\vee \cong \C[e^1] \to \Lambda_0$ 
which sends $e^1$ to $\lambda$.
The the curvature term and the differential 
 of the equivariant $A_\infty$ algebra are
\begin{align*}
\m_0^{\bb,\lambda}(1) &= T^{\omega(\beta)/2\pi} e^{\langle \bb,\partial\beta \rangle}
- \lambda i_{e_1}(\bb)\nonumber\\
&= T^uc_0(1+c_+) - \lambda \langle \bb , e_1 \rangle ,\\
\m_1^{\bb,\lambda}(\mathbf{e}^1) &=T^{\omega(\beta)/2\pi} e^{\langle \bb_0,\partial \beta \rangle}
\sum_{l_0+l_1=l} 
\m_{l+1,\beta}(\overbrace{\bb_+,...,\bb_+}^{l_0},e^1,\overbrace{\bb_+,...,\bb_+}^{l_1})
- \lambda i_{e_1}(\mathbf{e}^1)\nonumber\\
&= T^uc_0(1+c_+) - \lambda .
\end{align*}
Givental's potential function for $\C$
\begin{align}
F = T^ux-\lambda\log x = X - \lambda \log X + {\rm const.} \nonumber
\end{align}
is recovered by introducing the variables $x:=c_0(1+c_+)$ and $X:=T^ux$, 
and the differential $\m_1^{\bb,\lambda}$ vanishes 
when the derivative $\partial F = X - \lambda$ of $F$ equals zero.

Suppose $\lambda \neq 0$.
Then $F$ has a unique nondegenerate critical point $\m \in \crit F$.
If $\val(\lambda)>0$, $\m$ 
corresponds to $(L_u,\bb)$ with nonvanishing equivariant Floer cohomology, 
where $u=\val (\lambda)$, $L_u$ is the moment fiber over $u$, 
and $\bb=\bb_0+\bb_+$ is as in the case of $\C P^1$.

Then 
\begin{align}
2\m_2^{\bb,\lambda}(\mathbf{e}^1,\mathbf{e}^1) &= T^uc_0(1+c_+) 
= X = \partial^2 F \neq 0. \nonumber
\end{align}

Let $\F_\lag$ be the uncurved $A_\infty$ category 
with a single object $\lag = \{ (L_u,\bb=\bb_0+\bb_+) \}$.
Then we have the following as in the case of $\C P^1$.
\begin{theorem}[Homological mirror symmetry for $\C$]\label{maintheorem2}
Suppose $\lambda \neq 0$ and $\val(\lambda)>0$.
Then we have a cohomologically fully faithful $A_\infty$ functor
\begin{align*}
\phi \colon \F_\lag \to \br(F)_\infty = \mf(F_\m)_\infty
\end{align*}
by sending $(L_u,\bb)$ to $\Lambda_\m^\stab$, 
whose image split-generates $[\br(F)]$.
\end{theorem}

\section{Appendix A}
In this Appendix, we will compute the dimension of the Jacobian ring associated with
 a Landau-Ginzburg mirror of a smooth semi-projective toric variety (see Theorem \ref{maintheorem}).
\subsection{Preliminaries on polyhedrons} \label{polyhedron}
In \S \ref{polyhedron}, we introduce some notions of polyhedron, which is used throughout this appendix.
We mainly follow \cite{coxtor}.

Let $N \cong \Z^n$ be a free abelian group of rank $n$ and $M:=\Hom_{\Z}(N, \Z)$ be the dual lattice.
For a commutative ring $R$, set $N_R:=N \otimes_\Z R, M_R:=M \otimes_\Z R$.
For a convex polyhedral cone $\sigma \in N_\R,$ the dual cone $\sigma^\vee$ is defined by 
\[\sigma^\vee=\{u \in M_\R \mid \lan{u,v} \ge 0 \ \text{for all}\  v \in \sigma\}.\] 
Note that a convex polyhedral cone is full-dimensional if and only if its dual cone is strictly convex.
For convex polyhedral cones $\tau, \sigma \subseteq N_\R,$ we write $\tau \prec \sigma$ if $\tau$ is a face of $\sigma$.

%%%%%%%%%%%%%%%%%%%%%%%%%%%%%%%%%%%%%%%%%%%%%%%
Let $\Sigma$ be a fan in $N_\mathbb{R}$\  and 
$\varphi$ be a strictly convex support function on $\Sigma$ (see \cite[Definitions 3.1.2, 4.2.11, 6.1.12]{coxtor}).
In this appendix, we assume that 
\begin{align}
\text{the support} \ |\Sigma| \ \text{is a full-dimensional rational polyhedral cone.}
\end{align}
Let $v_1, v_2, \dots, v_m \in N$ be the set of integral generators of rays of $\Sigma.$

Let $P$ be a polyhedron in $M_\R$ determined by $\Sigma$ and $\varphi$, i.e., 
\[P=\{u \in M_\R\mid \ell_i(u):=\lan{u, v_i}+\varphi(v_i) \ge 0, \ i=1,2, \dots, m \}.\]
We call integral affine functions $\ell_i \ (i=1,2,\dots,m)$ the defining functions of $P$. 
For $v \in |\Sigma|$, we note that 
\begin{align*}\min_{u \in P} \lan{u, v}=-\varphi(v)\end{align*} .

Let $\uP$ be the recession cone of $P$, i.e., 
\[\uP:=\{u \in M_\R \mid \lan{u, v_i} \ge 0,\  \ i=1,2, \dots, m\}.\]
We note that 
\[\uP^\vee=|\Sigma|=\R_{\ge 0}v_1+\R_{\ge 0}v_2+\cdots +\R_{\ge 0}v_m.\] 
Since $|\Sigma|$ is a full-dimensional rational polyhedral cone, the polyhedron $P$ is pointed, i.e., $\uP$ is strictly convex. 

%%%%%%%%%%%%%%%%%%%%%%%               AFFINOID              %%%%%%%%%%%%%%%%%%%%%%%%%%%%%%%%%
\subsection{Preliminaries on polyhedral subdomains}\label{polyhedral subdomain}
We define a $\Lambda$-algebra $\Lambda\lan{U_P}$ by
\[\Lambda\lan{U_P}:=\left\{ \sum_{\ v \in N \cap |\Sigma|} c_v y^v
   \middle|c_v \in \Lambda, \lim_{|v| \to \infty}  (\lan{u, v}+\val(c_v))=\infty \ \text {for all} \ u \in P \right\},\]
where $|v|$ is a standard Euclidean norm on $N_\R$.
For $f=\sum c_v y^v \in \Lambda\lan{U_P},$ set 
	\[|f|:=\exp\big(-\!\!\!\inf_{\substack{u \in P, \\ v \in N \cap \uP^\vee}} (\lan{u, v}+\val(c_v))\big).\]
We note that $\inf_{u \in P}(\lan{u, v}+\val(c_v))=\val(c_v)-\varphi(v)$ for $v \in |\Sigma| \cap N$.	
Then $|\cdot|$ gives a complete non-archimedean norm on $\Lambda\lan{U_P}$
and the pair $(\Lambda\lan{U_P}, |\cdot|)$ is a $\Lambda$-Banach algebra (\cite[Remark 6.6]{rabtro}).
Moreover, we know that $(\Lambda\lan{U_P}, |\cdot|)$ is a $\Lambda$-affinoid algebra (\cite[Proposition 6.9]{rabtro}).

By construction, the affinoid algebra $\Lambda\lan{U_P}$ contains the monoid ring $\Lambda[\uP^\vee \cap N],$
which induces an inclusion  \[\Spm(\Lambda\lan{U_P}) \subseteq \Spm(\Lambda[\uP^\vee \cap N]),\]
where we denote by $\Spm$ the maximal spectrums (see \cite[Proposition 6.9]{rabtro}).

Let $\m \in \Spm(\Lambda\lan{U_P})$ and 
$\tilde{\m}:=\m \cap \Lambda[\uP^\vee \cap N]$ be the corresponding maximal ideal of $\Lambda[\uP^\vee \cap N]$.
Then we have an isomorphism between complete local rings
\[\widehat{\Lambda\lan{U_P}_\m} \cong \widehat{\Lambda[\uP^\vee \cap N]_{\tilde{\m}}}.\]
(See the proof of \cite[Proposition 6.9]{rabtro}.  See also \cite[Proposition 7.3.2.3]{bosnon} and \cite[Lemma 5.1.2]{conirr}.)
We note that $\Lambda[\uP^\vee \cap N]_{\tilde{\m}}$ is a Cohen-Macaulay ring of Krull dimension $n$ (see, e.g., \cite[Theorem 1]{hocrin}, \cite[\S 13.1]{eiscom}).
Hence $\Lambda\lan{U_P}_\m$ is a Cohen-Macaulay ring (\cite[Proposition 6.9]{rabtro}) of Krull dimension $n$.
Moreover, $\Lambda\lan{U_P}_\m$ is regular if and only if $\Lambda[\uP^\vee \cap N]_{\tilde{\m}}$ is regular.

%%%%%%%%%%%%%%%%%%%%%%%%%%%%%%%%              Jacobian ring         %%%%%%%%%%%%%%%%%%%%%%%%%%%%%%%%%%%%%%%%%%
For a $\Lambda$-Banach algebra $A$, set
\begin{align*}
A_0:=\{f \in A\mid |f| \le 1 \}, \ \ A_+:=\{f \in A\mid |f| < 1 \}, \ \ \overline{A}:=A_0/A_+.
\end{align*}
Then $A_0$ is a $\Lambda$-Banach algebra, $A_+$ is an ideal of $A_0$, and $\overline{A}$ is a $\C$-algebra.
For $v \in |\Sigma| \cap N$, let $\overline{y}^v \in \overline{\Lambda\lan{U_P}}$ be the residue class of 
$T^{\varphi(v)}y^v \in  \Lambda\lan{U_P}_0$.

For $c=\sum_{i=0}^\infty c_i T^{\lambda_i} \in \Lambda \setminus \{0\}$, the leading term coefficient $L(c) \in \C^*$ is defined by the coefficient of $T^{\val(c)}$ and 
set $\mathrm{Sp}(c):=\{\lambda_i\mid c_i \neq 0\} \subset \R.$
For $f=\sum_{v \in |\Sigma| \cap N} c_v y^v \in \Lambda\lan{U_P}$,
set $\mathrm{Sp}(f):=\bigcup_{v \in |\Sigma| \cap N}\mathrm{Sp}(T^{-\varphi(v)}c_v).$
Note that $\mathrm{Sp}(c)$ and $\mathrm{Sp}(f)$ are discrete subsets of $\R$. 
\begin{proposition}
We have \[\overline{\Lambda\lan{U_P}} \cong \bigoplus_{v \in |\Sigma| \cap N} \C \overline{y}^v\]
as $\C$-vector spaces.
The ring structure of $\overline{\Lambda\lan{U_P}}$ is given by 
\[\overline{y}^v \cdot \overline{y}^{v'}=
\begin{cases} 
\overline{y}^{v+v'} & \text{if}\ v, v' \in \sigma \ \text{for some}\ \sigma \in \Sigma\\
0 &\text{otherwise}. 
\end{cases}
\]
\end{proposition}
\begin{proof}
The isomorphism $\overline{\Lambda\lan{U_P}} \cong \bigoplus_{v \in |\Sigma| \cap N} \C \overline{y}^v$ easily follows from the definition.
Since $\varphi$ is strictly convex, we have the desired conclusion on the ring structure.
\end{proof}

Let $\Sym(M_\Lambda)$ be the symmetric algebra of the $\Lambda$-vector space $M_\Lambda$
and $\Lambda\lan{M}$ be the completion of the polynomial ring $\Sym(M_\Lambda)$ with respect to the Gauss norm so that $\Lambda\lan{M}$ is a Tate algebra.
Then we have $\overline{\Lambda\lan{M}} \cong \Sym(M_\C)$.
We see that $\Spm(\Lambda\lan{M})$ is naturally isomorphic to $N_{\Lambda_0}$.

For $f=\sum c_v y^v \in \Lambda\lan{U_P}$ and $\mr \in M$
we define a derivative $\partial_\mr f \in \Lambda \lan{U_P}$ by $\partial_\mr f:=\sum c_v \lan{\mr, v} y^v.$
We note that $\Lambda\lan{U_P}_0$, $\Lambda\lan{U_P}_+$ are closed under $\partial _\mr$ 
and $\partial_\mr$ also acts on $\overline{\Lambda\lan{U_P}}.$
We define a $\Lambda$-algebra homomorphism $\psi_f : \Lambda\lan{M} \to \Lambda\lan{U_P}$ by 
$\psi_f(\mr):=\partial_\mr f$.
If $f \in \Lambda\lan{U_P}_0$, %(resp. $\Lambda\lan{U_P}_+$), 
then we have $\psi_f(\Lambda\lan{M}_0) \subseteq \Lambda\lan{U_P}_0$
and $\psi_f(\Lambda\lan{M}_+) \subseteq \Lambda\lan{U_P}_+$.
Hence $\psi_f$ induces a $\C$-algebra morphism $\overline{\psi}_f : \Sym(M_\C) \to \overline{\Lambda\lan{U_P}}$.
\begin{proposition}\label{finite}
Let $f \in \Lambda\lan{U_P}_0$.
Suppose that $\overline{\Lambda\lan{U_P}}$ is finite over $\Sym(M_\C)$.
Then $\Lambda\lan{{U_P}}_0$ is finite over $\Lambda\lan{M}_0$.
\end{proposition}
\begin{proof}
Take $v_1, v_2, \cdots, v_r \in |\Sigma| \cap N$ such that $\overline{y}^{v_1}, \overline{y}^{v_2}, \cdots, \overline{y}^{v_r}$ 
generate $\overline{\Lambda\lan{U_P}}$ as a $\Sym(M_\C)$-module.
Take $g=\sum T^{\varphi(v)}c'_v y^v \in \Lambda\lan{U_P}_0$. 
Let $G$ be the discrete submonoid of 
$\R_{\ge 0}$ generated by \[\mathrm{Sp}(g) \cup \bigcup_{\substack{k \ge 1, \\ i=1, \dots, r}} \mathrm{Sp}(T^{\varphi(v_i)}f^ky^{v_i}).\]
We enumerate $G$ by 
\[G=\{c_0, c_1, c_2, \cdots \}, \ 0=c_0<c_1<c_2<\cdots.\]
Let $\overline{g}=\sum_{\val(c'_v)=0} L(c'_v)\overline{y}^v \in \overline{\Lambda\lan{U_P}}$ be the residue class of $g$.
Choose $\phi_{1, 0}, \cdots, \phi_{r, 0} \in \Sym(M_\C)$ such that 
$\overline{g}=\sum_{i=1}^r\phi_{i, 0}\overline{y}^{v_i}.$
Set $g_0=g$ and $g_1=g_0-\sum_{i=1}^rT^{c_0+\varphi(v_i)}\phi_{i, 0} y^{v_i}$, 
where $\phi_{i,0}$ are considered as elements of $\Lambda\lan{M}_0.$
Then we have $\mathrm{Sp}(g_1) \subset G \setminus\{c_0\}.$
Inductively, we can construct 
\[g_{k+1} \in \Lambda\lan{U_P}_0, \ \ \phi_{1, k}, \phi_{2, k}, \cdots, \phi_{r, k} \in \Sym(M_\C)\] 
with 
\[g_{k+1}=g_k-\sum_{i=1}^rT^{c_k+\varphi(v_i)}\phi_{i, k} y^{v_i}, \ \ \mathrm{Sp}(g_{k+1}) \subset G\setminus\{c_0, c_1, \cdots, c_k\}.\]
Set $\phi_i:=\sum_{j=0}^\infty T^{c_j}\phi_{i, j} \in \Lambda\lan{M}_0,$ then we have $g=\sum_{i=1}^r\phi_i T^{\varphi(v_i)}y^{v_i}$.
This implies that $T^{\varphi(v_i)} y^{v_i} (i=1, 2, \cdots, r)$ generate $\Lambda\lan{U_P}_0$ as a $\Lambda\lan{M}_0$-module.
\end{proof}
\begin{corollary}\label{flat}
Under the same assumptions of Proposition \ref{finite}, 
$\Lambda\lan{U_P}$ is flat over $\Lambda\lan{M}$.
\end{corollary}
\begin{proof}
Let $\m$ be a maximal ideal of $\Lambda\lan{U_P}$ and set $\m':=\psi_f^{-1}(\m),$ then $\m'$ is a maximal ideal of $\Lambda\lan{M}.$
By Proposition \ref{finite}, we easily see that $\Lambda\lan{U_P}$ is finite over $\Lambda\lan{M}$.
Combining with the incomparability (e.g., \cite[Corollary 4.18, Proposition 9.2]{eiscom}), the fiber $\Lambda\lan{U_P}/ \m'\Lambda\lan{U_P}$ over $\m'$ has Krull dimension $0$.
Since $\Lambda\lan{U_P}_\m$ is a Cohen-Macaulay local ring of Krull dimension $n$ and $\Lambda\lan{M}_{\m'}$ is a regular local ring of Krull dimension $n$, we see that  $\Lambda\lan{U_P}_\m$ is flat over $\Lambda\lan{M}_{\m'}$(e.g., \cite[Theorem 18.16]{eiscom}).
Thus $\Lambda\lan{U_P}$ is flat over $\Lambda\lan{M}$.
\end{proof}
In the last of \S \ref{polyhedral subdomain}, we introduce a logarithmic Jacobian ring.
\begin{definition} 
Let $f \in \Lambda\lan{U_P}_0$, $\lambda \in N_{\Lambda_0}$, and $\m_\lambda$ be the maximal ideal of $\Lambda\lan{M}$ corresponding to $\lambda$.
Set $\partial_{\mr}f_\lambda:=\partial_\mr f- \lan{\mr, \lambda}$.
A logarithmic Jacobian ideal $I_{f, \lambda}$ is defined by \[ I_{f, \lambda} :=\lan{\partial_\mr f_\lambda \mid \mr \in M},\]
and a logarithmic Jacobian ring $J_\lambda(f)$ is defined by $\Lambda\lan{U_P}/I_{f, \lambda}.$  
We note that $I_{f, \lambda}$ is the ideal generated by $\psi_f(\m_\lambda)$ and $\Spec(J_\lambda(f))$ is the fiber over $\lambda$.
\end{definition}
%\begin{remark}
%Suppose that $|f| \le 1$, then $|\partial_\mr f| \le 1$ for all $\mr \in M$.
%Choose an integral basis $e_1, e_2, \dots e_n$ of $M$, then we can define a continuous ring morphism 
%\[\varphi_f :\Lambda\lan{\lambda_1, \lambda_2, \dots, \lambda_n} \to \Lambda\lan{U_P}\] by 
%$\varphi_f(\lambda_i):=\partial_{e_i}(f).$
%If $\lambda$ is contained in $N_{\Lambda_0}$, 
%then $\lambda$ is naturally considered as an element of $\Spm(\Lambda\lan{\lambda_1,\lambda_2, \dots \lambda_n})$.
%Let $\m$ be the maximal ideal corresponding to $\lambda \in N_{\Lambda_0},$
%then the ideal generated $\varphi_f(\m)$ is equal to $I_{f, \lambda},$ i.e., $\Spec(J_\lambda(f))$ is the fiber over $\lambda.$
%It may be interesting problem to show that $\varphi_f$ is finite. 
%\end{remark}

%%%%%%%%%%%%%%%%%%%%%%%%%%%%               Change Ssunsection                         %%%%%%%%%%%%%%%%%%%%%%%%%%%%%%%
\subsection{Preliminaries on tropicalizations}\label{subtrop}
Let $\overline{\R}$ be the additive monoid $\R \cup \{\infty\}$ 
and let $\sigma \subseteq N_\R$ be a rational full-dimensional convex polyhedral cone.
We define $M_\R(\sigma^\vee)$ by the space of monoid morphisms 
\[M_\R(\sigma^\vee):=\Hom_{\R_{\ge0}}(\sigma, \overline{\R}) \]
respecting multiplication by $\R_{\ge 0}.$
Since $\sigma$ is rational, we easily see that
\[\Hom_{\R_{\ge0}}(\sigma, \overline{\R}) \cong \Hom(\sigma \cap N,  \overline{\R}),\]
where $\Hom$ is the space of monoid morphisms.
Let $\tau$ be a face of $\sigma$, $u \in M_\R/\tau^\perp$ and $v \in \sigma.$
We define $\iota(u) \in M_\R(\sigma^\vee)$ by
 \[\lan{\iota(u), v}:=
 \begin{cases}\lan{u, v} \ &\text{if} \ v \in \tau\\
                  \infty \ &\text{otherwise}. 
 \end{cases}\]  
Then this correspondence gives an isomorphism 
\[\iota : \coprod_{\tau \prec \sigma} M_\R/ \tau^\perp \xrightarrow{\simeq} M_\R(\sigma^\vee)\]
(see \cite[Proposition 3.4]{rabtro}).
Since $\sigma$ is full-dimensional, $M_\R$ is naturally contained in $M_\R(\sigma^\vee)$. 
For $\m \in \Spm(\Lambda[\sigma \cap N])$, by using $\Spm(\Lambda[\sigma \cap N]) \cong \Hom(\sigma \cap N, \Lambda),$
we denote by $\phi_\m$ the corresponding element of $\Hom(\sigma \cap N, \Lambda)$. 
We define a tropicalization morphism 
\[\trop : \Spm(\Lambda[\sigma \cap N]) \to M_\R(\sigma^\vee)\]
by 
$\trop(\m):=\val \circ \phi_\m \in M_\R(\sigma^\vee).$

%%%%%%%%%%%%%%%%%%%%%%%%%%%%%              Compactification of polyhedron              %%%%%%%%%%%%%%%%%%%%%%%%%%%%%%%%%%%
Recall that  $P \subseteq M_\R$ is a polyhedron determined by $\Sigma$ and $\varphi$ with a strictly convex recession cone $\uP.$
We define $\overline{P} \subseteq M_\R(\uP)$ as follows:
Let $\tau$ be a face of $\uP^\vee$ and $u \in M_\R/\tau^\perp$.
Then $\iota(u) \in \overline{P}$ if and only if
\[\ell_i(u)=\lan{u, v_i}+\varphi(v_i) \ge 0 \ \text{for all}\  v_i \in \tau.\]

Set \[ U_P:=\{\m \in \Spm(\Lambda[\uP^\vee \cap N])\mid \trop(\m) \in \overline{P}\},\]
then we see that $\Spm(\Lambda\lan{U_P})=U_P \subseteq \Spm(\Lambda[\uP^\vee \cap N])$ (\cite[Proposition 6.9]{rabtro}).
The restriction of the tropicalization morphism $\trop$ to $U_P$ is also denoted by $\trop$.
\begin{definition}
Let $\aaa$ be an ideal of $\Lambda\lan{U_P}$ and set $V(\aaa):=\Spm(\Lambda\lan{U_P}/\aaa) \subseteq U_P$.
Then the image $\trop(V(\aaa)) \subseteq \overline{P}$ is called a tropicalization of $V(\aaa)$ and denoted by $\Trop(\aaa).$
If $\aaa$ is generated by $f \in \Lambda \lan{U_P}$, $V(\aaa)$ is also denoted by $V(f)$ and its tropicalization is denoted by $\Trop(f)$.
\end{definition}

%%%%%%%%%%%%%%%%%%%%%%%                      Tropicalization                               %%%%%%%%%%%%%%%%%%%%%%%%%%%%%%%%%%%%%
For $c \neq 0 \in \Lambda,$ recall that the leading term $L(c) \in \C^*$ is the coefficient of $T^{\val(c)}.$
For $f=\sum c_vy^v \in \Lambda\lan{U_P}, \tau \prec \uP^\vee,$ and $u \in (M_\R/\tau^\perp) \cap \overline{P}$, 
set \[m:=\min_{v \in \tau \cap N} \big(\lan{u, v}+\val(c_v)\big).\] 
We define the initial term $\ini_u(f)$ by 
\[\ini_u(f):=\sum_{\substack{v \in \tau \cap N \\ \lan{u, v}+\val(c_v)=m}}L(c_v)y^v \in \C[\tau \cap N].\] 
By \cite[Lemma 8.4]{rabtro}, we have 
\begin{align}\label{hyp}
\Trop(f)=\{u \in \overline{P} \mid \ini_u(f) \ \text{is not a monomial}\}.
\end{align} 
\begin{remark}
It is possible that $\ini_u(f)=0.$
If $\ini_u(f)=0$, then $\ini_u(f)$ is not considered as a monomial.
\end{remark}
By the fundamental theorem of tropical geometry\cite[Theorem 7.8]{rabtro}, we have
\begin{align}\label{fundamental}
\Trop(\aaa)=\bigcap_{f \in \aaa}\Trop(f).\end{align}

%%%%%%%%%%%%%%%%%%%%%%%%%%%%                Change Subsection                              %%%%%%%%%%%%%%%%%%%%%%%%%%%%%%%%
 \subsection{Main theorem}\label{main}
 In \S \ref{main}, we assume that the fan $\Sigma$ is smooth.
 Let \[\Lambda\lan{Z_1,Z_2, \dots, Z_m}\] be the Tate algebra over $\Lambda$ with variables $Z_1, Z_2, \dots, Z_m$
 (the affinoid algebra associated with the polyhedron $\R^m_{\ge 0}).$  
 Let \[\psi : \Lambda\lan{Z_1,Z_2, \dots, Z_m} \to \Lambda\lan{U_P}\]
 be the continuous morphism which is defined by 
 $\psi(Z_i)=T^{\varphi(v_i)}y^{v_i}.$
 Surjectivity of $\psi$ follows from the smoothness of $\Sigma$. 
 Let \[\overline{F}=c_1 Z_1+c_2Z_2+\cdots+c_mZ_m \in \Lambda\lan{Z_1, Z_2, \dots, Z_m},\] where $c_i \in \C^*$
 and let $F$ be an element of $\Lambda\lan{Z_1, Z_2, \dots, Z_m}$ with $|F-\overline{F}|<1$.
 Set $\overline{f}=\psi(\overline{F})$ and $f=\psi(F)$.
 Take $\lambda \in N_{\Lambda_0}$.
 The next statement is the main theorem of this appendix.
 \begin{theorem}\label{maintheorem}
 Let $\Sigma$ be a smooth fan such that the support $|\Sigma|$ is a full dimensional rational polyhedral cone,
 $\lambda \in N_{\Lambda_0}$, and $f$ is as avobe.
 Then $\dim_\Lambda J_\lambda(f)=\dim_\C H^*(X_\Sigma ; \C).$ 
 \end{theorem}
 To prove this theorem, we first show the next proposition.
 \begin{proposition}\label{independent}
 $\dim_\Lambda J_\lambda(f)$ is independent of $\lambda \in N_{\Lambda_0}.$
 \end{proposition}
 \begin{proof}
 We recall that $f$ induces a $\C$-algebra morphism $\overline{\psi}_f : \Sym(M_\C) \to \overline{\Lambda\lan{U_P}}.$
 Using this morphism, we consider $\overline{\Lambda\lan{U_P}}$ as a module over $\Sym(M_\C).$
 Set $\overline{y'}^{v_i}:=c_i\overline{y}^{v_i}.$
 By definition, for $\mr \in M_\C$, we have $\overline{\psi}_f(\mr)=\sum_{i=1}^m \lan{\mr, v_i}\overline{y'}^{v_i}.$
 Let $X_\Sigma$ be the $n$-dimensional smooth toric variety  associated with the fan $\Sigma.$
 Then the torus $T:=N\otimes_\Z \C^*$ naturally acts on $X_\Sigma.$
 We naturally identify the equivariant cohomology $H^*_T(\mathrm{pt})$ with $\Sym(M_\C)$ as $\C$-algebras.
 We easily see that $\overline{\Lambda\lan{U_P}}$ is isomorphic to 
 the equivariant cohomology $H^*_T(X_\Sigma; \C)$ as $\Sym(M_\C)$-algebras (see, e.g., \cite[\S 12.4]{coxtor}).
 Since $H^*_T(X_\Sigma; \C)$ is a free module over $\Sym(M_\C)$ of finite rank (e.g., \cite[Proposition 2.1]{irishi} and \cite[\S 7.2]{coxtor}), 
 the $\Lambda\lan{M}$-module $\Lambda\lan{U_P}$ is finite and flat by Proposition \ref{finite} and Corollary \ref {flat}.
 Hence $\Lambda\lan{U_P}$ is a finite locally free $\Lambda\lan{M}$-module and dimensions of fibers are independent of 
 $\lambda \in N_{\Lambda_0}.$
 %see e.g., stack project \S 29.48 and reference therein.  
 \end{proof}
 We next compute $\dim_\Lambda J_\lambda(f)$ at a specific point $\lambda \in N_{\Lambda_0}$ (Theorem \ref{main2}).
 \begin{lemma}
 Let $\mr \in M, \tau \prec |\Sigma|$ and $u \in M_\R/\tau^\perp$.
 Then $u \in \Trop(\partial_\mr f_\lambda)$ if and only if $u \in \Trop(\partial_\mr \overline{f}_{\lambda})$.
 \end{lemma}
 \begin{proof}
 For $G \in \Lambda\lan{Z_1,Z_2,\dots, Z_m}$, we set \[\partial_\mr G:=\sum_{i=1}^m\lan{\mr, v_i}Z_i \frac{\partial G}{\partial{Z_i}}.\]
 Then we see that $\partial_\mr G \in \Lambda\lan{Z_1,Z_2,\dots, Z_m}$ and $\psi(\partial_\mr G)=\partial_\mr\psi(G)$.
 Choose a monomial $Z^{\bv}:=Z_1^{b_1} Z_2^{b_2} \cdots Z_m^{b_m}$, where $b_1, b_2, \dots, b_m \in \Z_{\ge 0}$. 
 Set 
 \[v_{\bv}:=b_1v_1+b_2v_2+\cdots+b_mv_m, \ \varphi_{\bv}:=b_1\varphi(v_1)+b_2\varphi(v_2)+\cdots+b_m \varphi(v_m).\]
 Then we have
 \begin{align*}
 \psi( \partial_\mr Z^{\bv})=\lan{\mr, v_{\bv}}T^{\varphi_{\bv}} y^{v_{\bv}}.
 \end{align*}
 Suppose that $v_{\bv} \in \tau$ and $\partial_\mr Z^{\bv} \neq 0$,  which implies $b_i=0$ for $v_i \notin \tau$.
 Then there exists $v_j \in \tau$ with $\lan{\mr, v_j} \neq 0$ and $b_j \neq 0$.
 We see that 
 \[\lan{u, v_{\bv}}+\varphi_{\bv}=\sum_{v_i \in \tau} b_i \ell_i(u) \ge \ell_j(u).\]
 Since $|F-\overline{F}| < 1$, the valuation of each coefficient of $F-\overline{F}$ is positive.
 Hence $Z^{\bv}$-term of $F-\overline{F}$ does not contribute to $\Trop(\partial_\mr f_\lambda)$. 
 \end{proof}
 Combining this lemma with Equations (\ref{hyp}) and (\ref{fundamental}), it follows that 
 \begin{align}
 \Trop\big(J_{\lambda}(f)\big) \subseteq \bigcap_{\mr \in M}\Trop(\partial_\mr \overline{f}_{\lambda}). \nonumber
 \end{align}
 For each maximal cone $\sigma \in \Sigma(n)$,  let $e_1^\sigma, e_2^\sigma, \dots, e_n^\sigma$ be the integral basis of $N$ which   
 generate $\sigma$ and let $f_1^\sigma, f_2^\sigma, \dots f_n^\sigma$ be the dual basis. 
 Set \[\lambda_i^\sigma:=\lan{f_i^\sigma, \lambda} \in \Lambda_0.\]
         
 For each positive dimensional cone $\tau  \in \Sigma\setminus \{\{0\}\}$ and 
 $\epsilon \in \R_{>0},$ set 
 \begin{align*}I_\tau &:=\{v_i \mid \tau \ \text{and}\  v_i \ \text{do not span any cone in} \ \Sigma \},\\
                 P_\tau&:=\{u \in P \mid \ell_i(u) \le \ell_j(u) \ \text{for all} \ v_i \in \tau \ \text{and} \ v_j \in I_\tau\}, \\
                 V_{\tau, \epsilon}&:=\{u \in P \mid \ell_i(u) \le \epsilon \ \text{for all}\ v_i \in \tau\}.\end{align*}
 For $\tau \in \Sigma \setminus \{\{0\}\}$ with $I_\tau \neq \emptyset$, we also set         
 \begin{align*}
 \epsilon_\tau&:=\inf_{v_i \in I_\tau, u \in P_\tau}\ell_i(u), \\
 \epsilon_P&:=\min\{\epsilon_\tau \mid \tau \in \Sigma \setminus \{\{0\}\} \text{ with } I_\tau \neq \emptyset \}.
 \end{align*}
 We easily see that $\epsilon_P>0$ and $V_{\tau, \epsilon_\tau} \subseteq P_\tau$.
 \begin{theorem}\label{main2}
 Suppose that $\epsilon_P>\val(\lambda_i^\sigma)>0$ for all $\sigma \in \Sigma(n)$ and $i$.
 For $\sigma \in \Sigma(n)$, we define $u_\lambda^\sigma \in \inte P$ by $\ell_i(u)=\val(\lambda_i ^\sigma) \ (v_i \in \sigma)$.
 Then 
 \[\Trop(J_\lambda(f))=\{u_\lambda^\sigma\mid \sigma \in \Sigma(n)\}\]
 and, for each $\sigma \in \Sigma(n)$, 
 there exists a unique critical point $\m_\lambda^\sigma \in \Spm(J_\lambda(f))$ with $\trop(\m^\sigma_\lambda)=u_\lambda^\sigma$. 
 Moreover, these critical points $\m^\sigma_\lambda$ are non-degenerate.  
 \end{theorem}
 \begin{proof}
 Let $u \in \Trop(J_\lambda(f))$.
 We first show that $u \in P$.
 Suppose that $u \in M_\R /\tau^\perp$ for some proper face $\tau \prec |\Sigma|$.
 Take $\mr \neq 0 \in \tau^\perp$ with $\lan{\mr, \lambda} \neq 0.$ 
 Then only the constant term $\lan{\mr, \lambda}$ contributes to $\ini_u(\partial_\mr \overline{f}_\lambda)$. 
 This contradicts $u \in \Trop(J_\lambda(f))$.

 Suppose that $u \in P \cap \Trop(J_\lambda(f))$. 
 Choose $v_{i_1}$ with $\ell_{i_1}(u) \le \ell_i(u)$ for all $i=1,2, \dots, m$ and let $\tau_1$ be the cone in $\Sigma$ spaned by $v_{i_1}.$
 By considering $\Trop(\partial_{f_i^\sigma} \overline{f}_\lambda)$ for some $\sigma \in \Sigma(n)$ and $i$ with $\lan{f^\sigma_i, v_{i_1}}=1$,
 we see that $\ell_{i_1}(u)  \le \val(\lambda_i^\sigma)< \epsilon_P \le \epsilon_{\tau_1}$, which implies $u \in P_{\tau_1}$.
 Choose $v_{i_2}$ with $\ell_{i_2}(u) \le \ell_i(u)$ for all $v_i \notin I_{\tau_1} \cup \{v_{i_1}\}$.
 We note that $\epsilon_{\tau_1} \le \ell_i(u)$ for all $v_i \in I_{\tau_1}$.
 By considering $\Trop(\partial_{f_i^\sigma} \overline{f}_\lambda)$ for some $\sigma \in \Sigma(n)$ and $i$ with 
 $\lan{f^\sigma_i, v_{i_1}}=0, \lan{f^\sigma_i, v_{i_2}}=1$, we see that $\ell_{i_2}(u)  \le \val(\lambda_i^\sigma)< \epsilon_P.$
 Inductively, we can construct a maximal cone $\sigma_u \in \Sigma(n)$ such that $\ell_i(u)<\epsilon_P$ for all $v_i \in \sigma_u$.
 We note that $\ell_i(u) \ge \epsilon_P$ if $v_i \notin \sigma_u$.
 We can assume $\lan{f_i ^{\sigma_u}, v_j}=\delta_{i,j} \ (i,j=1,2, \dots n)$.
 Then we have 
 \[\trop(\partial_{f_i^{\sigma_u}}\overline{f}_\lambda)=\{ u \in P \mid \ell_i(u)=\val(\lambda_i^{\sigma_u})\},\]
 which implies that \[\Trop(J_\lambda(f)) \subseteq \{u_\lambda^\sigma\mid \sigma \in \Sigma(n)\}.\]
 On the other hand, 
 by \cite[Theorem 11.7]{rabtro}, 
 it follows that for each $\sigma \in \Sigma(n)$ there exists a unique critical point $\m \in \Spm(J_\lambda(f))$ with    $\trop(\m)=u^{\sigma}_\lambda$ and these critical points are non-degenerate.
 \end{proof}
 Combining Proposition \ref{independent} and Theorem \ref{main2}, we complete the proof of Theorem \ref{maintheorem}.

%%%%%%%%%%%%%%%%%%%%%         MATRIX FACTORIZATION         %%%%%%%%%%%%%%%%%%%%%%%%%%%%%%%
 \section{Appendix  B}\label{variant}
 In \S \ref{variant}, we introduce a variant of a category of matrix factorizations of Givental type potentials (Definition \ref{variantmf}).
In this section, we assume that the base field $k$ is equal to $\C$.
Let $R, f, g_1, g_2, \dots g_r$ be the same as \S \ref{givental}, i.e., $R$ is a commutative regular  $\C$-algebra of Krull dimension $n$,
$f \in R$, and $g_1, g_2, \dots, g_r $ are invertible elements of $R.$
 Set
\[T:=R/\lan{g_1-1, g_2-1, \dots, g_r-1}.\]
Suppose that $dg_1, dg_2, \dots, dg_r$ are linearly independent at each point of $\Spec(T) \subseteq \Spec(R)$, which implies that $T$ is also regular.
We denote by $f_T$ the restriction of $f$ to $\Spec(T)$.
We will give a relationship between $[\overline{\br(f_T)}]$ and a matrix factorization category for a Givental type potential function.

%Set $\RR:=R[\lambda_1, \lambda_2, \dots, \lambda_r]$, where $\lambda_1, \lambda_2, \dots, \lambda_r$ are considered as variables.
Let $X$ be the analytification of $\Spec(R)$, $Y \subseteq X$ be the analytification of $\Spec(T)$, and $\C^r_\lambda$ be the $r$-dimensional complex plane with coordinates $\lambda_1, \lambda_2, \dots, \lambda_r$.
For $\epsilon \in \R_{>0}$, set \[U:=\{x \in X \mid |g_1(x)-1|<\epsilon,|g_2(x)-1|<\epsilon, \dots, |g_r(x)-1|<\epsilon\}.\]
By taking $\epsilon$ enough small, we can choose the branch of $\log g_i$ on $U$ such that $\log g_i =0$ along $Y$. 
We define $F^\an$ by 
\[F^\an:=f-\lambda_1 \log g_1-\lambda_1 \log g_2- \cdots - \lambda_r \log g_r.\]
This is an element of the ring $\ano_{U \times \C^r_\lambda}$ of holomorphic functions on $U \times \C^r_\lambda$.
Let $\pr_1$ be the projection from $X \times \C^r_\lambda$ to $X$ and $i$ be the inclusion of $X \cong X \times \{0\}$ to $X \times \C^r_\lambda$.
By the method of Lagrange multipliers, for $y \in \crit(f_T)$, there exists a unique critical point $L(y) \in \crit(F^\an)$ with $\pr_1(L(y))=y$
and this map $L$ gives an isomorphism $\crit(f_T) \cong \crit(F^\an)$.
For a complex manifold $Z$ with a holomorphic function $h \in \ano_Z$, a point $p \in Z$,
 and a coordinate system $z_1, z_2, \dots z_n$ near $p$,  we define the Jacobian algebra and the Tyurina algebra of $h$ at $p$ by 
\begin{align*}
J(h)_p:=&\ano_{Z, p}/\lan{\partial_{z_1}h, \partial_{z_2}h, \dots, \partial_{z_n}h}, \\
T(h)_p:=&\ano_{Z, p}/\lan{h, \partial_{z_1}h, \partial_{z_2}h, \dots, \partial_{z_n}h},
\end{align*}
where $\ano_{Z, p}$ is the ring of analytic germs of the holomorphic functions at $p$.
The formal Taylor expansion of $h$ at $p$ is denoted by $\hat{h}_p$.
This is an element of the formal completion $\hat{\ano}_{Z, p}.$

For $y \in \crit (f_T)$, take an open neighborhood $V \subseteq U \subseteq X$ of $y$ and a coordinate system $t_1, t_2, \dots, t_n$ on $V$ with $t_1=\log g_1, t_2=\log g_2, \dots, t_r=\log g_r$ and $t_1(y)=\cdots=t_n(y)=0.$ 
Using this coordinate system, we easily see that 
\[J(F^\an)_{L(y)} \cong J(f_T)_y.\]
Hence $L(y)$ is an isolated singular point of $F^\an$ if and only if $y$ is an isolated singular point of $f_T$.
For an isolated singular point $y \in \crit(f_T)$, we will show that 
\[[\mf(\hat{f}_{T, y})] \cong [\mf(\hat{F}^\an_{L(y)})].\]
Using the coordinate system $t_1, t_2, \dots, t_n$ near $y \in \crit(f_T)$,
we define $F^\an_T \in \ano_{V \times \C^r_\lambda, i(y)}$ by 
\[f_T-\lambda_1 t_1-\lambda_2 t_2-\cdots-\lambda_r t_r.\]
By easy computation, we see that 
\[T(F^\an)_{L(y)} \cong T(F^\an_T)_{i(y)} \cong T(f_T)_y.\]
By a theorem of Mather and Yau \cite{matcla}, we see that 
\[\ano_{V \times \C^r_\lambda, L(y)}/ F^\an \cong \ano_{V \times \C^r_\lambda, i(y)}/ F^\an_T.\]
Combining with \cite[Theorem 2.10]{orlfor}, we have
\[[\mf(\hat{F}^\an_{L(y)})] \cong [\mf(\hat{F}_{T, i(y)})].\]
On the other hand, by the Kn\"{o}rrer periodicity (\cite[Theorem 3.1]{knocoh}, see also \cite[\S 2.1]{dyccom}), we have
\[[\mf(\hat{F}_{T, i(y)})] \cong [\mf(\hat{f}_{T, y})].\]
Thus we see that   
\[[\mf(\hat{f}_{T, y})] \cong [\mf(\hat{F}^\an_{L(y)})].\]

\begin{definition}\label{variantmf}
Suppose that $f_T$ has isolated singularities.
We define a triangulated category of matrix factorizations of $F^\an$ by
\[\overline{[\br^{I\hspace{-.1em}I}(F^\an)]}:=\prod_{p \in \crit(F^\an)}[\mf(\hat{F}^\an_p)]\] 
\end{definition}
\begin{proposition}[cf. {\cite[Theorem 1.2]{hirder}}]
Suppose that $f_T$ has isolated singularities.
Then we have 
\[\overline{[\br^{I\hspace{-.1em}I}(F^\an)]} \cong \overline{[\br(f_T)]}.\]
\end{proposition}

\end{document}